\documentclass[11pt,leqno,a4paper,twoside]{article}

\usepackage[english]{babel}
\usepackage{latexsym,amsfonts,amsmath,amsthm,amstext,amscd,amssymb,euscript,wasysym, color}
\usepackage{rotating}
\usepackage{multirow}
\usepackage{array}

\usepackage{enumitem}
\usepackage{graphicx}
\usepackage{subfigure}
\usepackage{verbatim}
\usepackage{epstopdf}
\usepackage{wasysym}
\usepackage{leftidx}

\usepackage{srcltx}

\usepackage{color}

\makeatletter
\@addtoreset{equation}{section}
\makeatother

\makeatletter
\@addtoreset{enunciato}{section}
\makeatother

\newcounter{enunciato}[section]

\newtheorem{ittheorem}{Theorem}
\newtheorem{itlemma}{Lemma}
\newtheorem{itproposition}{Proposition}
\newtheorem{itcorollary}{Corollary}

\theoremstyle{definition}
\newtheorem{itdefinition}{Definition}
\newtheorem{itremark}{Remark}
\newtheorem{note}{}
\newtheorem{itexample}{Example}

\newenvironment{theorem}{\addtocounter{enunciato}{1}
\begin{ittheorem}}{\end{ittheorem}}

\newenvironment{lemma}{\addtocounter{enunciato}{1}
\begin{itlemma}}{\end{itlemma}}

\newenvironment{definition}{\addtocounter{enunciato}{1}
\begin{itdefinition}}{\end{itdefinition}}

\newenvironment{corollary}{\addtocounter{enunciato}{1}
\begin{itcorollary}}{\end{itcorollary}}

\newenvironment{example}{\addtocounter{enunciato}{1}
\begin{itexample}}{\end{itexample}}

\newcommand{\cA}{\mathcal{A}}
\newcommand{\cB}{\mathcal{B}}
\newcommand{\cC}{\mathcal{C}}

\newcommand{\cL}{\mathcal{L}}

\newcommand{\cN}{\mathcal{N}}

\newcommand{\cX}{\mathcal{X}}
\newcommand{\cY}{\mathcal{Y}}

\newcommand{\fC}{\mathfrak{C}}

\newcommand{\fJ}{\mathfrak{J}}

\newcommand{\fS}{\mathfrak{S}}

\newcommand{\N}{\mathbb{N}}

\newcommand{\R}{\mathbb{R}}

\newcommand{\1}{\mathbf{1}}

\DeclareMathOperator{\D}{\, \mathrm{d}\!}

\DeclareMathOperator{\supp}{supp}

\newcommand{\Ds}{\frac{d}{d s}}

\begin{document}

\title{Gibbs-non-Gibbs dynamical transitions \\
for mean-field interacting Brownian motions}

\author{
\renewcommand{\thefootnote}{\arabic{footnote}}
F. den Hollander
\footnotemark[1]
\\
\renewcommand{\thefootnote}{\arabic{footnote}}
F. Redig
\footnotemark[2]
\\
\renewcommand{\thefootnote}{\arabic{footnote}}
W. van Zuijlen
\footnotemark[1]
}

\footnotetext[1]{
Mathematical Institute, Leiden University, P.O.\ Box 9512, 2300 RA, Leiden, The
Netherlands.
}
\footnotetext[2]{
Delft Institute of Applied Mathematics, Delft University of Technology, Mekelweg 4,
2628 CD Delft, The Netherlands.
}

\maketitle

\begin{abstract}
\noindent
We consider a system of real-valued spins interacting with each other
through a mean-field Hamiltonian that depends on the empirical
magnetisation of the spins. The system is subjected to a stochastic
dynamics where the spins perform independent Brownian motions.
Using large
deviation theory we show that there exists an explicitly computable
crossover time $t_c \in [0,\infty]$ from Gibbs to non-Gibbs. We give
examples of immediate loss of Gibbsianness ($t_c=0$), short-time
conservation and large-time loss of Gibbsianness
($t_c\in (0,\infty)$), and preservation of Gibbsianness ($t_c=\infty$).
Depending on the potential, the system can be Gibbs or non-Gibbs
at the crossover time $t=t_c$.

{\it Acknowledgment.} FdH and WvZ are supported by ERC Advanced Grant VARIS-267356.
The authors are grateful to A.\ van Enter, R.\ Fern\'andez and J.\ Mart\'{i}nez for discussions.
\end{abstract}


\section{Introduction and main results}
\label{S1}


\subsection{Background}
\label{S1.1}

Gibbs states are mathematical tools to describe physical interacting particle systems.
In the lattice context, a Gibbs measure is a probability measure on the configuration space
where the conditional distributions inside a finite subset of the lattice, 
given that the configuration outside this set is fixed, are described by a Gibbs 
specification, i.e., by a Boltzmann factor depending on an absolutely summable 
interaction potential (see Georgii \cite[Definition 2.9]{Ge88}).
When such systems evolve over time according to a stochastic dynamics, it may
happen that the time-evolved state no longer is Gibbs. This phenomenon was originally
discovered and described for heating dynamics by van Enter, Fern\'andez, den Hollander
and Redig~\cite{vEFedHoRe02}. In this paper, a low-temperature Ising model is subjected
to a high-temperature Glauber spin-flip dynamics. The state remains Gibbs for short times,
but becomes non-Gibbs after a finite time. If the magnetic field is zero, then Gibbsianness
once lost is never recovered.  But if the magnetic field is non-zero and small enough, then
Gibbsianness is recovered at later times.

By now results of this type are available for a variety of interacting particle systems, both
in the lattice setting and in the mean-field setting. Both for heating dynamics and for cooling
dynamics estimates are available on transition times, as well as characterisations of the
so-called \emph{bad configurations} leading to non-Gibbsianness (i.e., the ``points of essential
discontinuity of the conditional probabilities''). It has become clear that Gibbs-non-Gibbs
transitions are the rule rather than the exception. We refer the reader to the recent overview
by van Enter \cite{vE12}.

In many papers non-Gibbsianness is proved by looking at
the evolving system at two times, the initial time and the final time, and applying techniques
from equilibrium statistical mechanics. This is a \emph{static} approach that does not illuminate
the relation between the Gibbs-non-Gibbs phenomenon and the dynamical effects responsible
for its occurrence.
This unsatisfactory situation was addressed in Enter, Fern\'andez, den
Hollander and Redig~\cite{vEFedHoRe10}, where possible dynamical mechanisms were
proposed and a \emph{program} was put forward to develop a theory of Gibbs-non-Gibbs
transitions in terms of \emph{large deviations for trajectories of relevant physical quantities}.

Fern\'andez, den Hollander and Mart\'inez~\cite{FedHoMa13}, \cite{FedHoMapr}, building on
earlier work by K\"ulske and Le Ny~\cite{KuLeNy07} and Ermolaev and K\"ulske~\cite{ErKu10},
showed that this program can be fully carried out for the Curie-Weiss model of Ising spins
subjected
 to an infinite-temperature spin-flip dynamics, and also for a Kac-type version of the
Curie-Weiss model. The present paper extends these works to systems of continuous
spins that interact with each other through a \emph{general} mean-field interaction potential and
perform independent \emph{Brownian motions}. The fact that we consider Brownian motions
allows us to obtain a \emph{complete characterisation} of passages from Gibbs to non-Gibbs.
The key notions of interest are \emph{good magnetisations} and \emph{bad magnetisations}
in the thermodynamic limit. Gibbsianness corresponds to having only good magnetisations,
while non-Gibbsianness corresponds to having at least one bad magnetisation.


\subsection{Outline}
\label{S1.2}

The definition of Gibbs for mean-field models differs from that for lattice models because
the interaction depends on the size of the system and does not have a geometric structure. In Section~\ref{S1.3} we introduce the
notions of a sequence of finite-volume mean-field Gibbs measures with a potential, good magnetisations, bad magnetisations 
and \emph{sequentially Gibbs}, and show that a sequence of finite-volume mean-field Gibbs measures with a continuously differentiable potential is sequentially Gibbs. In Section~\ref{S1.4} we define the 
Brownian motion dynamics. We show that a magnetisation $\alpha \in \R$ is bad at time 
$t$ if and only if the large deviation rate function for the magnetisation at time $0$ conditional 
on the magnetisation at time $t$ being $\alpha$ has multiple global minimisers. We further 
show that the system is sequentially Gibbs at time $t$ if and only if all magnetisations are 
good at time $t$. In Section~\ref{S1.5} we show that a magnetisation $\alpha$ is bad at time 
$t$ if and only if the large deviation rate function for the \emph{trajectory} of the magnetisation 
conditional on hitting the value $\alpha$ at time $t$ has multiple global minimisers. We further 
show that different minimising trajectories are different at time $0$. In Section~\ref{S1.6} we 
show that Gibbsianness can be classified in terms of the \emph{second difference quotient 
of the potential}. With the help of this classification we show that there exists a unique time 
$t_c \in [0,\infty]$ at which the system changes from Gibbs to non-Gibbs, and give a characterisation 
of $t_c$ in terms of the potential associated with the starting measures. In Section~\ref{S1.7} 
we give examples for which $t_c=0$, $t_c \in (0,\infty)$ and $t_c=\infty$. In Section~\ref{S1.8} 
we discuss our results and indicate possible future research. Proofs are given in 
Sections~\ref{s:speckernel}--\ref{s:multglmin}. Appendix~\ref{appendix} collects a few key 
formulas that are needed along the way. Appendix~\ref{appendix2} contains some background on proper weakly continuous regular conditional probabilities. 


\subsection{Sequences of finite-volume mean-field Gibbs measures, Potential, Sequentially Gibbs}
\label{S1.3}

In this section we give the definition of a sequence of finite-volume mean-field Gibbs measures
(Definition~\ref{def:mean-field_Gibbs_sequence}), and of good/bad
magnetisations and sequentially Gibbs sequences
(Definition~\ref{def:good,bad_magnetisation_sequentially_gibbsianness}).
We show that a sequentially Gibbs sequence has a weakly continuous
specification kernel (Lemma~\ref{lemma:gibbsian_sequences_have_sort_of_specification_kernel}).
We show that sequences of finite-volume mean-field Gibbs measures with a continuously differentiable potential are sequentially Gibbs
(Theorem~\ref{theorem:for_a_continuously_differentiable_potential_with_certain_properties_we_get_a_Gibbs_sequence}).

In what follows, we write $\N=\{1,2,3,\dots\}$ and $\N_{\ge 2}=\N\setminus\{1\}$. For
$n\in\N$, $\cB(\R^n)$ denotes the Lebesgue measurable subsets of $\R^n$, and
$\mu_{\cN(v,A)}$ denotes the normal distribution on $\cB(\R^n)$ with mean vector
$v \in \R^n$ and covariance matrix $A\in \R^{n\times n}$.
We write $I_n$ for the identity matrix in $\R^{n\times n}$.
For $\alpha\in\R$ and
$\epsilon>0$, $B(\alpha,\epsilon)$ denotes the open ball of radius $\epsilon$ centered at
$\alpha$.

\begin{definition}
\label{def:mean-field_Gibbs_sequence}
For $n\in\N$, the \emph{empirical magnetisation} $m_n\colon\,\R^n \rightarrow \R$ is
given by
\begin{flalign}
&&m_n(x_1,\dots,x_n)= \frac1n \sum_{i=1}^n x_i
&&\big((x_1,\dots,x_n)\in\R^n\big).
\end{flalign}
For $n\in\N$, let $\nu_n$ be a probability measure on $\cB(\R^n)$. Let $V\colon\, \R \to [0,\infty)$
be a Borel measurable function. The sequence $(\nu_n)_{n\in\N}$ is called a \emph{sequence of finite-volume mean-field Gibbs measures} 
with \emph{potential} $V$ and \emph{reference measures} $(\mu_{\cN(0,I_n)})_{n\in\N}$
when
\begin{flalign}
\label{eq:mfg}
&&\nu_n(A) = \frac{1}{Z_n} \int_{\R^n} \1_A(x)\, e^{-n \left(V \circ\, m_n\right)(x)} \D \mu_{\cN(0,I_n)}(x)
&& (A\in \cB(\R^n), n\in\N),
\end{flalign}
where $Z_n\in (0,\infty)$ is the \emph{normalizing constant}.
\end{definition}

\noindent
Note that $\nu_n$ in \eqref{eq:mfg} does not change when $V$ is replaced by $V+c$ for some $c\in \R$.
Therefore our assumption that $V \geq 0$ is equivalent to the assumption that $V$ is bounded from below. 

The model described in Definition \ref{def:mean-field_Gibbs_sequence} is an example of a mean-field model, where the Hamiltonian ($H_n(x)= n (V \circ m_n)(x)$) depends on the magnetisation ($m_n(x)$) only. 
In general the Hamiltonian of a mean-field model depends on the empirical mean (i.e., on ``$\frac{1}{n} \sum_{i=1}^n \delta_{x_i}$''), but we restrict ourselves to the models in Definition \ref{def:mean-field_Gibbs_sequence}. 

\begin{definition}
\label{def:good,bad_magnetisation_sequentially_gibbsianness}
For $n\in\N$, let $\rho_n$ be a probability measure on $\cB(\R^n)$, and let $\pi_{(2:n)}\colon\,
\R^n \rightarrow \R^{n-1}$ be defined by
\begin{flalign}
&&\pi_{(2:n)}(y_1,\dots,y_n)= (y_2,\dots,y_n)
&&\big((y_1,\dots,y_n)\in \R^n\big).
\end{flalign}
Suppose that for every $n\in \N_{\ge 2}$ there exists a weakly continuous proper regular conditional
probability  $\gamma_n \colon\, \R^{n-1} \times \cB(\R) \rightarrow [0,1]$ under $\rho_n$ of the first
spin given the other spins, i.e., $\gamma_n$ is the unique weakly continuous probability kernel for
which 
\begin{flalign}
&\rho_n \left(A\times B\right)
= \int_{\R^n} \1_B(y_2,\dots,y_n)\, \gamma_n\big((y_2,\dots, y_n), A\big)
\D\, \big[\rho_n\circ \pi_{(2:n)}^{-1}\big](y_2,\dots,y_n)
\hspace{-4cm} \\ \nonumber
& &&(A\in \cB(\R),\,B\in \cB(\R^{n-1})).
\end{flalign}
See Appendix~\ref{appendix2} for precise definitions and properties of these objects.
\begin{enumerate}
\item[{\rm (a)}]
$\alpha\in \R$ is called a \emph{good magnetisation} for the sequence $(\rho_n)_{n\in\N}$
when there exists a probability measure $\gamma_\alpha\colon\,\cB(\R) \rightarrow [0,1]$ for which
the sequence of measures $(\gamma_n(v_{n-1},\cdot))_{n\in\N_{\ge 2}}$ weakly converges
to $\gamma_\alpha$ for all sequences $(v_{n-1})_{n\in\N_{\ge 2}}$ with $v_{n-1} \in \R^{n-1}$
for which the empirical magnetisation of $v_n$ converges to $\alpha$, i.e., $m_{n-1}(v_{n-1})
\rightarrow \alpha$.
\item[{\rm (b)}]
$\alpha \in \R$ is called a \emph{bad magnetisation} when it is not a good magnetisation.
\item[{\rm (c)}]
The sequence $(\rho_n)_{n\in\N}$ is called \emph{sequentially Gibbs} when all $\alpha \in \R$
are good magnetisations.
\end{enumerate}
\end{definition}

\noindent 
The notion of Gibbs for a mean-field model was introduced by K\"ulske and Le 
Ny~\cite[Definition 2.1]{KuLeNy07} (see also K\"ulske~\cite{Ku03}) and is the same as our definition of sequentially Gibbs
(even though our definition of good magnetisation is slightly different).

The following lemma shows that, in the thermodynamic limit $n\to\infty$, the probability measure
of the first spin given the magnetisation of the other spins is a transition kernel that depends weakly continuously
on the magnetisation of the other spins. This lemma will be proved in Section~\ref{s:speckernel}.

\begin{lemma}
\label{lemma:gibbsian_sequences_have_sort_of_specification_kernel}
Let $(\rho_n)_{n\in\N}$ be sequentially Gibbs. With the same notation as in
Definition~{\rm \ref{def:good,bad_magnetisation_sequentially_gibbsianness}},
define $\gamma\colon\,\R \times \cB(\R)\rightarrow [0,1]$ by letting $\gamma(\alpha, \cdot)
= \gamma_\alpha$. Then  $\alpha \mapsto \gamma(\alpha,\cdot)$ is weakly continuous and,
consequently, $\gamma$ is a transition kernel (called the \emph{specification kernel}).
\end{lemma}

Our first main result, whose proof will be given in Section~\ref{section:proof_theorem_initial_seq_is_Gibbs},
shows that a sequence of finite-volume mean-field Gibbs measures  with a continuously  differentiable potential is sequentially Gibbs.

\begin{theorem}
\label{theorem:for_a_continuously_differentiable_potential_with_certain_properties_we_get_a_Gibbs_sequence}
Let $(\nu_n)_{n\in\N}$ be a sequence of finite-volume mean-field Gibbs measures with potential $V\colon\, \R \rightarrow [0,\infty)$.
\begin{enumerate}
\item[{\rm (a)}]
Define $\overline \gamma_n\colon\, \R \times \cB(\R) \rightarrow [0,1]$  by
\begin{flalign}
\label{eqn:formula_for_overline_gamma_n,t}
&&\overline \gamma_n(\alpha, A) = \frac{\int_\R \1_A(x) e^{-nV(\frac{n-1}{n}\alpha + \frac{x}{n})}
e^{-x^2/2} \D x}{\int_\R e^{-nV(\frac{n-1}{n}\alpha + \frac{x}{n})} e^{-x^2/2} \D x}
&&(\alpha \in \R, A\in \cB(\R)).
\end{flalign}
Then $\gamma_n\colon\, \R^{n-1} \times \cB(\R) \rightarrow [0,1]$ defined by $\gamma_n(v,A) = \overline
\gamma_n(m_{n-1}(v),A)$ for $v\in \R^{n-1}$ and $A\in \cB(\R)$ is the weakly continuous proper conditional
probability under $\nu_n$ of the first spin given the other spins.
\item[{\rm (b)}]
If $V$ is continuously differentiable on a neighbourhood of $\alpha\in \R$, then $\overline \gamma_n(\alpha_n, \cdot)$ converges weakly (even strongly)  to
$\mu_{\cN(-V'(\alpha),1)}$ for all sequences $(\alpha_n)_{n\in\N}$ that converge to
$\alpha$ (in particular, $\alpha$ is a good magnetisation for $(\nu_n)_{n\in\N}$).
\item[{\rm (c)}]
If $V$ is continuously differentiable, then $(\nu_n)_{n\in\N}$ is sequentially Gibbs.
\end{enumerate}
\end{theorem}

In Section \ref{S1.7} we give an example of a non-differentiable potential for which 
$(\nu_n)_{n\in\N}$ is a sequence of finite-volume mean-field Gibbs measures but not  sequentially Gibbs (Example \ref{example:mfg_not_seq_gibbs}, where we write $\mu_{n,0}$ instead of $\nu_n$).


\subsection{Brownian motion dynamics}
\label{S1.4}

In this section we introduce the Brownian motion dynamics, give the essential tools for
identifying good magnetisations (Lemma~\ref{lemma:goodiden}) and global minimisers
of a certain tilted form of the potential (Lemma~\ref{lemma:main_lemma_about_equivalences}),
and show that a magnetisation is good if and only if the tilted potential has a unique global
minimiser (Theorem~\ref{theorem:sequentially_gibbs_iff_unique_minimiser}).

For $n\in\N$, $\mu_{n,0}$ represents the law of the $n$ spins at time $t=0$. 
We assume that $(\mu_{n,0})_{n\in\N}$ is a sequence of finite-volume mean-field Gibbs measures with potential $V$. 
Let $\mu_{n,t}$
be the evolved law at time $t\in (0,\infty)$ when the $n$ spins perform independent Brownian
motions, i.e.,
\begin{flalign}
\label{eqn:formula_for_mu_n,t}
&& \mu_{n,t}(A) = \frac{1}{Z_n} \int_{\R^n} p_n(t,z,A)\, e^{-n (V \circ\, m_n)(z) }
\D \mu_{\cN(0,I_n)}(z)
&& (A\in \cB(\R^n))
\end{flalign}
(recall \eqref{eq:mfg}), where
\begin{flalign}
\label{eqn:formula_for_P_n}
&& p_n(t,z,A) = \mu_{\cN(z,tI_n)}(A)= (2\pi t)^{- \frac{n}{2}} \int_{\R^n}
\1_A(y)\, e^{-\frac{\|y-z\|^2}{2t}} \D y
&& (z\in \R^n, A\in \cB(\R^n)).
\end{flalign}
There exists a weakly continuous proper regular conditional probability $\gamma_{n,t}$ under $\mu_{n,t}$ of the first spin given the other spins for which $\gamma_{n,t}(u,\cdot) = \gamma_{n,t}(v,\cdot)$ for all $u,v\in \R^{n-1}$ with $m_{n-1}(u)=m_{n-1}(v)$ (a proof and an expression for $\gamma_{n,t}$ are given in Appendix~\ref{appendix}). Therefore
we can determine whether or not $(\mu_{n,t})_{n\in\N}$ is sequentially Gibbs by looking
at the sequence $(\overline \gamma_{n,t})_{n\in\N}$ of probability kernels $\R \times
\cB(\R) \rightarrow [0,1]$, where $\overline \gamma_{n,t}(\alpha,\cdot) = \gamma_{n,t}(v,\cdot)$
for all $v\in \R^{n-1}$ and $\alpha \in \R$ with $m_{n-1}(v) = \alpha$ (an expression for
$\overline \gamma_{n,t}$ is given in Appendix~\ref{appendix} and also in \eqref{eqn:overline_gamma_n,t_in_terms_of_g_nt}). This is formalized in the
following lemma.

\begin{lemma}
\label{lemma:goodiden}
Let $t\in (0,\infty)$. Then $\alpha\in\R$  is a good magnetisation
for $(\mu_{n,t})_{n\in\N}$ if and only if there exists a measure $\gamma_\alpha\colon\,\cB(\R)
\to [0,1]$ such that the sequence $(\overline \gamma_{n,t}(\alpha_n,\cdot))_{n\in\N}$ converges
weakly to $\gamma_\alpha$ for all sequences $(\alpha_n)_{n\in\N}$ in $\R$ that converge to
$\alpha$.
\end{lemma}

The function $\eta_{n,t}\colon\, \R \times \cB(\R) \rightarrow [0,1]$
defined for $n\in\N$ and $t\in (0,\infty)$ by
\begin{flalign}
\label{eqn:formula_for_eta_nt}
&&\eta_{n,t}(\alpha,A)
& = \frac{
\int_\R \1_A(s)\,  e^{-n \big[V(s) + \frac{s^2}{2} + \frac{(s-\alpha)^2}{2t}\big]} \D s}
{\int_\R  e^{-n \big[V(s) + \frac{s^2}{2} + \frac{(s-\alpha)^2}{2t}\big]}\D s}
&& (\alpha\in\R, A\in \cB(\R)).
\end{flalign}
is the weakly continuous proper regular conditional
probability of the magnetisation at time $0$ given the magnetisation at time $t$
(see Appendix~\ref{appendix}).
By den Hollander \cite[Theorem III.17]{dH00}, the sequence
$(\eta_{n,t}(\alpha, \cdot))_{n\in\N}$ satisfies the large deviation principle with rate $n$
and rate function
\begin{align}
\label{eqn:rate_function_of_eta_n,t}
r\mapsto V(r) + \frac{r^2}{2} + \frac{(r-\alpha)^2}{2t} - \inf_{s\in \R}
\Big[V(s) + \frac{s^2}{2} + \frac{(s-\alpha)^2}{2t}\Big].
\end{align}
(See Dembo and Zeitouni \cite{DeZe10} or den Hollander \cite{dH00} for background on large deviations.)
With this notation, $\overline \gamma_{n,t}$ can be written as (see Appendix~\ref{appendix})
\begin{flalign}
\label{eqn:overline_gamma_n,t_in_terms_of_g_nt}
&& \overline \gamma_{n,t}(\alpha,B)
& = \frac{\int_\R \mu_{\cN(s,t)}(B) \ g_{n,t}(\alpha,s)
\D \mu_{\cN(0,1)}(s)}{ \int_\R g_{n,t}(\alpha,s) \D \mu_{\cN(0,1)}(s)}
&& (\alpha \in \R, B\in \cB(\R)),
\end{flalign}
where $g_{n,t}\colon\, \R^2 \rightarrow \R$ is given by
\begin{flalign}
\label{eqn:definition_g_n,t}
&& g_{n,t}(\alpha,s)
& = \frac{ \int_\R e^{-n\left[ V(r+ \frac1n (s-r)) - V(r)\right]}\, e^{-V(r)}
\D \left[\eta_{n-1,t}(\alpha,\cdot)\right](r) }{\int_\R e^{-V(r)} \D \left[ \eta_{n-1,t}(\alpha,\cdot)\right](r) }
&& (\alpha,s\in \R).
\end{flalign}

The following lemma will be proved in Section~\ref{proofs_of_bifurcation}.

\begin{lemma}
\label{lemma:main_lemma_about_equivalences}
Let $V\in C^1(\R,[0,\infty))$, $t\in (0,\infty)$ and $\alpha \in \R$.
\begin{enumerate}
\item[{\rm (a)}]
If \eqref{eqn:rate_function_of_eta_n,t} has a unique global minimiser $q\in \R$, then
there exists a $\mu_{\cN(0,1)}$-integrable function $h\colon\,\R \rightarrow [0,\infty)$ such
that
\begin{flalign}
\label{gntcond}
&& g_{n,t}(\alpha_n,s) \rightarrow e^{-sV'(q)}  &&(s\in \R),\\
&& g_{n,t}(\alpha_n,s) \le h(s)  &&(n\in\N,s\in \R), \nonumber
\end{flalign}
for  all sequences $(\alpha_n)_{n\in\N}$ that converge to $\alpha$.
\item[{\rm (b)}]
Let $q_1,q_2$ be the smallest, respectively, the largest global minimiser of
\eqref{eqn:rate_function_of_eta_n,t}. Then there exists a $\mu_{\cN(0,1)}$-integrable function
$h\colon\, \R \rightarrow [0,\infty)$, and sequences $(\alpha_n^1)_{n\in\N}$ and $(\alpha_n^2)_{n\in\N}$
both converging to $\alpha$, for which \eqref{gntcond} holds with $q=q_1$, $\alpha_n = \alpha_n^1$ and with $q= q_2$, $\alpha_n = \alpha_n^2$, respectively. 
\end{enumerate}
\end{lemma}

\noindent In case \eqref{eqn:rate_function_of_eta_n,t} has multiple minimisers, Lemma \ref{lemma:main_lemma_about_equivalences}(b) implies that there are sequences $(\alpha_n^1)_{n\in\N}$ and $(\beta_n^2)_{n\in\N}$ that in some sense ``select'' the smallest and the largest global minimiser of \eqref{eqn:rate_function_of_eta_n,t}, respectively. In the proof of Lemma \ref{lemma:main_lemma_about_equivalences} we will see that this is the case for $\alpha_n^1 = \alpha - \frac{1}{\sqrt{n}}$ and $\alpha_n^2 = \alpha + \frac{1}{\sqrt{n}}$. 

Our second main result shows that sequentially Gibbs is equivalent to uniqueness of the
global minimiser of \eqref{eqn:rate_function_of_eta_n,t}.

\begin{theorem}
\label{theorem:sequentially_gibbs_iff_unique_minimiser}
Let $V\in C^1(\R,[0,\infty))$. Then for every $t\in (0,\infty)$
\begin{enumerate}
\item[{\rm (a)}]
$\alpha\in \R$ is a good magnetisation for $(\mu_{n,t})_{n\in\N}$ if and only if
\eqref{eqn:rate_function_of_eta_n,t} has a unique global minimiser.
\item[{\rm (b)}] 
If $\alpha\in \R$ is a good magnetisation for $(\mu_{n,t})_{n\in\N}$, then
\begin{flalign}
&& \gamma_\alpha(B) = \mu_{\cN(-V'(q),1+t)}(B) && (B\in \cB(\R)),
\end{flalign} 
where $\gamma_\alpha$ is the (limiting) probability measure as in 
\emph{Definition \ref{def:good,bad_magnetisation_sequentially_gibbsianness}(a)}.
\item[{\rm (c)}]
$(\mu_{n,t})_{n\in\N}$ is sequentially Gibbs if and only if \eqref{eqn:rate_function_of_eta_n,t}
has a unique global minimiser for all $\alpha \in \R$.
\end{enumerate}
\end{theorem}

\noindent
The claim in Theorem~\ref{theorem:sequentially_gibbs_iff_unique_minimiser}(a) follows from
Lemma~\ref{lemma:main_lemma_about_equivalences},
\eqref{eqn:overline_gamma_n,t_in_terms_of_g_nt} and Lebesgue's Dominated Convergence Theorem
after we note that if $q_1,q_2\in\R$ with $q_1\ne q_2$ are global minimisers of
\eqref{eqn:rate_function_of_eta_n,t},  then $V'(q_1)-V'(q_2)=(q_2-q_1)(1+t^{-1})\not=0$. 
By Lemma \ref{lemma:main_lemma_about_equivalences}(a),
\begin{flalign}
&& \gamma_\alpha(B)= \frac{\int_\R \mu_{\cN(s,t)}(B) e^{-sV'(q)}
\D \mu_{\cN(0,1)}(s)}{ \int_\R e^{-sV'(q)}\D \mu_{\cN(0,1)}(s)} && (B\in \cB(\R)),
\end{flalign}
from which it is easy to conclude Theorem \ref{theorem:sequentially_gibbs_iff_unique_minimiser}(b). Theorem \ref{theorem:sequentially_gibbs_iff_unique_minimiser}(c) is an immediate consequence of Theorem \ref{theorem:sequentially_gibbs_iff_unique_minimiser}(a).


\subsection{Trajectories of the magnetisation (Intermezzo)}
\label{S1.5}

In this section we consider the probability measure on the set of trajectories of the magnetisation
between time $0$ and time $t$. 
We show the equivalence of uniqueness of the minimising magnetisation
at time $0$ and uniqueness of the minimising trajectory of the magnetisation
(Theorem~\ref{theorem:rate_path_conditional_seq_and_comparing_with_two_layer}).
This characterises good and bad magnetisations in terms of the trajectory of the magnetisation
(Corollary~\ref{corollary:sequentially_gibbs_iff_unique_minimiser_path_and_two_layer}). 

By considering minimising trajectories instead of minimising initial points of the magnetisation, 
we obtain a better picture of the effects of the evolution. 
The name \emph{two-layer model} has been used for a description of the minimisation problem 
for the magnetisation at time $0$ given the magnetisation at time $t$.
As Section 1.6 will confirm, the optimisation problem for the two-layer model is computationally easier. 
However, in contrast with obtaining the function \eqref{eqn:rate_function_of_eta_n,t}, obtaining the large deviation rate function for the two-layer model for more general dynamics, e.g., independent diffusion processes, might not be so easy and the rate function might not be given by an explicit formula like \eqref{eqn:rate_function_of_eta_n,t}. 
For example, we took advantage of the fact that the transition kernel for the Brownian motion 
over time $t$ is given explicitly. 
For more general diffusions this is not the case and we expect it to be necessary to consider the large deviation rate function for the trajectories (with the goal to obtain an implicit formula for the large deviation rate function for the two-layer model in terms of the more explicit large deviation rate function for the trajectories, by means of the contraction principle). 
We will show that for the case of independent Brownian motions the minimising problem for the two-layer model and the minimising problem for the trajectories are equivalent by showing that the minimising paths for the trajectories are fully determined by their initial point and endpoint.

Let $\mu_n$ be the law on $C([0,\infty),\R^n)$ of the paths of the independent Brownian motions performed by the $n$ spins with initial distribution $\mu_{n,0}$. Thus, with $P(x,\cdot)$ denoting
the law of the Brownian motion on $C([0,\infty),\R)$ starting at $x\in \R$ and $\fS_{C([0,\infty),\R^n)}$ denoting the
Skorohod $\sigma$-algebra on $C([0,\infty),\R^n)$, we have
\begin{flalign}
\label{eqn:formula_for_mu_n}
&& \mu_n(A) & = \int_{\R^n}  \left(\bigotimes_{i=1}^n P(x_i,\cdot)\right)(A) \D \mu_{n,0}(x_1,\dots,x_n)
&& (A \in  \fS_{C([0,\infty),\R^n)}).
\end{flalign}
Let $t\in (0,\infty)$. Let $Q_{n,t}\colon\, \R \times \fS_{C([0,t],\R)} \rightarrow [0,1]$ be the transition
kernel where $Q_{n,t}(s,\cdot)$ is the probability measure of a Brownian motion with variance
$\frac{1}{n}$ starting at $s$. We write $m_n$ also for the function $C([0,t],\R^n) \rightarrow C([0,t],\R)$ given by
\begin{flalign}
&& m_n \big(\phi_1,\dots, \phi_n\big) = \frac1n \sum_{i=1}^n \phi_i,
&& \Big( (\phi_1,\dots,\phi_n)\in C([0,t],\R^n) \Big).
\end{flalign}
Then $Q_{n,t}(s,A)= [ \! \otimes_{i=1}^n P(x_i, \cdot)] (\pi_{[0,t]}^{-1} (m_n^{-1}(A)))$ for 
all $A\in \fS_{C([0,t],\R)}$ and all $s\in \R$ and $x=(x_1,\dots,x_n)\in \R^n$ with $m_n(x) =s$, where $\pi_{[0,t]}\colon\, C([0,\infty), \R^n) \rightarrow C([0,t],\R^n)$
is given by $\pi_{[0,t]}(\phi) = \phi|_{[0,t]}$. We have
\begin{flalign}
\label{eqn:connection_conditional_initial_vs_path}
&& \mu_n\circ \pi_{[0,t]}^{-1}\left( m_n^{-1}(A) \right)
= \int_\R  Q_{n,t}(s, A) \D \left[ \mu_{n,0}\circ m_n^{-1}\right] (s)
&& (A \in  \fS_{C([0,t],\R)}).
\end{flalign}
Let $\pi_t\colon\, C([0,t],\R) \rightarrow \R$ be the projection on the endpoint of the path, i.e., $\pi_t(\phi)
= \phi(t)$.

\begin{theorem}
\label{theorem:rate_path_conditional_seq_and_comparing_with_two_layer}
Let $t\in (0,\infty)$. 
\begin{enumerate}
\item[{\rm (a)}]
For every $n\in\N$ there exists a weakly continuous proper regular conditional probability $\rho_n\colon\,
\R \times \fS_{C([0,t),\R)}\rightarrow [0,1]$ under $\mu_{n} \circ \pi_{[0,t]}^{-1} \circ m_n^{-1}$ given $\pi_t$
(given the endpoint of the trajectory). 
\item[{\rm (b)}]
 For all $\alpha \in \R$, $(\rho_n(\alpha,\cdot))_{n\in\N}$ satisfies the
large deviation principle (in $C([0,t),\R)$ equipped with the uniform topology) with rate $n$ and rate function $C([0,t),\R) \rightarrow [0,\infty]$ given by
\begin{flalign}
\label{eqn:rate_function_paths}
\phi\mapsto
\begin{cases}
V(\phi(0)) + \frac12 \phi(0)^2 + \frac12 \int_0^t \dot \phi(s)^2 \D s - C_{t,\alpha},
& \mbox{if } \phi\in \cA\cC([0,t),\R) \mbox{ and } \lim_{s\uparrow t}\phi(s)=\alpha, \\
\infty, & otherwise,
\end{cases}
\end{flalign}
where $\cA\cC([0,t),\R)$ is the set of absolutely continuous functions from $[0,t]$ to $\R$ restricted to $[0,t)$, and $C_{t,\alpha} = \inf_{s_0 \in \R} V(s_0) + \frac{s_0^2}{2}  + \frac{(\alpha- s_0)^2}{2t} $. 
\item[{\rm (c)}] For every $\alpha \in \R$, \eqref{eqn:rate_function_paths} has a unique global minimiser
if and only if  \eqref{eqn:rate_function_of_eta_n,t}  has a unique global minimiser.
\end{enumerate}
\end{theorem}

\begin{proof}
The proof of (a) and (b) is given in Appendix \ref{appendix}.
For (c) we only prove the `if' implication. The function 
\begin{align}
\label{eqn:integral_of_derivative_squared}
\cL^1([0,t],\R) \rightarrow [0,\infty], \qquad
g \mapsto \int_0^t g^2(s) \D s
\end{align}
is strictly convex (on $\cL^2([0,t],\R)$), since $2ab< a^2 + b^2$ for $a,b\in \R$ with $a\ne b$.
Hence, for all $r\in \R$, the path $\psi(s) = r + \frac{r-\alpha}{t}s$ for $s\in [0,t]$ is the unique path that minimises 
\begin{align}
\label{eqn:infimum_over_schilder_ldp}
 \inf_{\phi \in \cA\cC([0,t],\R), \phi(0)=r, \phi(t)=\alpha} \tfrac12 \int_0^t \dot \phi^2(s) \D s.
\end{align}
In particular \eqref{eqn:infimum_over_schilder_ldp} equals $\frac{(r-\alpha)^2}{2t}$. Hence
the infimum of \eqref{eqn:rate_function_paths} over all paths $\phi \in C([0,t),\R)$ with $\phi(0)=r$ is equal to \eqref{eqn:rate_function_of_eta_n,t}.
\end{proof}

As a consequence of Theorem~\ref{theorem:rate_path_conditional_seq_and_comparing_with_two_layer},
we can refine the result of Theorem~\ref{theorem:sequentially_gibbs_iff_unique_minimiser}.

\begin{corollary}
\label{corollary:sequentially_gibbs_iff_unique_minimiser_path_and_two_layer}
Let $V\in C^1(\R,[0,\infty))$. Then for every $t\in (0,\infty)$:
\begin{enumerate}
\item[{\rm (a)}]
For $\alpha \in \R$ the following  are equivalent: \\
{\rm (a1)}
$\alpha\in \R$ is a good magnetisation for $(\mu_{n,t})_{n\in\N}$, \\
{\rm (a2)}
\eqref{eqn:rate_function_of_eta_n,t} has a unique global minimiser, \\
{\rm (a3)}
\eqref{eqn:rate_function_paths} has a unique global minimiser.
\item[{\rm (b)}]
The following  are equivalent:\\
{\rm (b1)}
$(\mu_{n,t})_{n\in\N}$ is sequentially Gibbs,\\
{\rm (b2)}  \eqref{eqn:rate_function_of_eta_n,t}
has a unique global minimiser for all $\alpha \in \R$,\\
{\rm (b3)} \eqref{eqn:rate_function_paths}
has a unique global minimiser for all $\alpha \in \R$.
\end{enumerate}
\end{corollary}


\subsection{Uniqueness of the minimisers of the rate function}
\label{S1.6}

In this section we give a necessary and sufficient condition in terms of the second
difference quotient of $V$ (Definition~\ref{def:secdifquo}) to have uniqueness
of the global minimisers of \eqref{eqn:rate_function_of_eta_n,t}
(Theorem~\ref{theorem:rate_function_has_multiple_global_minima_expressed_in_second_differential_quotient}
and Corollary~\ref{corollary:overview_Gibbs_non_Gibbs}). From this condition it follows
that Gibbsianness can never be recovered once it is lost.

\begin{definition}
\label{def:secdifquo}
Let $f\colon\, \R \rightarrow \R$. The \emph{second difference quotient} of $f$ is the function
\begin{align}
&\Phi_2 f\colon\, \{ (x,y,z)\in \R^3\colon\, x<y<z\}  \rightarrow \R, \\ \notag
&(x,y,z) \mapsto \frac{1}{z-x} \left( \frac{f(z)-f(y)}{z-y} - \frac{f(y)-f(x)}{y-x} \right).
\end{align}
\end{definition}

Our third main result, whose proof will be given in Section \ref{s:multglmin}, is the following
classification of Gibbsianness.

\begin{theorem}
\label{theorem:rate_function_has_multiple_global_minima_expressed_in_second_differential_quotient}
Let $V: \R \rightarrow [0,\infty)$ be lower semicontinuous. Fix $t\in (0,\infty)$. There exists an $\alpha \in \R$ for which \eqref{eqn:rate_function_of_eta_n,t} has
multiple global minimisers if and only if $\Phi_2 V \not > - \frac{1+t}{2t}$, i.e., if and only if there exist
$a,b,c\in \R$ with $a<b<c$ for which $\Phi_2 V(a,b,c) \le - \frac{1+t}{2t}$. Consequently, there exists
a crossover time $t_c\in [0,\infty]$ such that $(\mu_{n,t})_{n\in\N}$ is sequentially Gibbs for $t\in (0,t_c)$
and not sequentially Gibbs for $t\in (t_c,\infty)$.
\end{theorem}

At  $t= t_c$, $(\mu_{n,t})_{n\in\N}$ may be sequentially Gibbs or not sequentially Gibbs. Both
scenarios are possible (see Example \ref{example:other_potentials_t_c_bigger_zero}).
Theorem~\ref{theorem:sequentially_gibbs_iff_unique_minimiser}(c) together with 
Theorem~\ref{theorem:rate_function_has_multiple_global_minima_expressed_in_second_differential_quotient}
yield the following.

\begin{corollary}
\label{corollary:overview_Gibbs_non_Gibbs}
Let $V\in C^1(\R,[0,\infty))$. Fix $t\in (0,\infty)$. For all $\alpha \in \R$, \eqref{eqn:rate_function_of_eta_n,t} has a unique global
minimiser if and only if $\Phi_2 V> - \frac{1+t}{2t}$. Consequently, the following scenarios occur
(where $M \in (\tfrac12,\infty)$):
\begin{enumerate}
\item[{\rm (a)}]
$(\mu_{n,t})_{n\in\N}$ is sequentially Gibbs\\[0.1cm]
{\rm (a1)} for $t\in (0,\infty)$ when $\Phi_2 V \geq - \frac12$,\\ [0.1cm]
{\rm (a2)} for $t \in (0,(M-\tfrac12)^{-1})$ when $\Phi_2 V \geq - M$,\\[0.1cm]
{\rm (a3)} for $t \in (0,(M-\tfrac12)^{-1}]$ when $\Phi_2 V > - M$.
\item[{\rm (b)}]
$(\mu_{n,t})_{n\in\N}$ is not sequentially Gibbs\\[0.1cm]
{\rm (b1)} for $t \in ((M-\tfrac12)^{-1},\infty)$ when $\Phi_2 V \not \ge -M$,\\[0.1cm]
{\rm (b2)} for $t \in [(M-\tfrac12)^{-1},\infty)$ when $\Phi_2 V \not > -M$,\\[0.1cm]
{\rm (b3)} for $t\in (0,\infty)$ when $\Phi_2 V$ is not bounded from below.
\end{enumerate}
\end{corollary}

Note that if $V$ is convex, then $(\mu_{n,t})_{n\in\N}$ is sequentially Gibbs for all
$t\in (0,\infty)$. We will see at the end of Section~\ref{s:multglmin} that if $V\in
C^2(\R,[0,\infty))$, then (a1),(a2) and (b1),(b2) hold with $\Phi_2 V$ replaced by $V''$.


\subsection{Examples}
\label{S1.7}

In this section we give examples of continuously differentiable potentials for each of the scenarios described in
Corollary~\ref{corollary:overview_Gibbs_non_Gibbs} (Examples~\ref{example:poleven}--\ref{example:other}).

\begin{example}
\label{example:poleven}
{\bf [Polynomial potentials: $t_c \in (0,\infty]$, sequentially Gibbs at $t=t_c$]}\\
Let $m\in\N$, $a_{2m} \in (0,\infty)$, $a_{2m-1},\dots,a_2,a_1\in \R$. Let $a_0\in \R$ be such that
\begin{align}
V\colon\, \R \rightarrow \R, \ r\mapsto a_{2m} r^{2m} + a_{2m-1} r^{2m-1} + \cdots a_1 r^1 + a_0
\end{align}
satisfies $V \geq 0$.
Since $V''$ is a polynomial of even degree, it is
bounded from below, say $V'' \ge -M$ for some $M\in (0,\infty)$. Hence, if $V$ is such a polynomial, then
the crossover time $t_c$ is strictly positive, i.e., $t_c\in (0,\infty]$. For example, for the potentials $V(r)= 0$, $V(r) = r^2$ and
$V(r) = r^4 - \frac12 r^2 + 1$, $(\mu_{n,t})_{n\in\N}$ is sequentially Gibbs for all $t\in [0,\infty)$, while for the
potentials $V(r) = r^4 - 4r^2 + 3$  and $V(r) = (r^2 - 9)^2$ there exists a $t_c\in (0,\infty)$ for which
$(\mu_{n,t})_{n\in\N}$ is sequentially Gibbs for $t\in [0,t_c)$ and not sequentially Gibbs for $t\in (t_c,\infty)$.

If $m=1$, then $\Phi_2 V = a_1 >0$. Hence $t_c = \infty$ by Corollary \ref{corollary:overview_Gibbs_non_Gibbs}.

If $m\ge 2$, then $V''$ is a polynomial of even degree  at least $ 2$.
Hence, if $\beta= -\tfrac12\inf_{r\in \R} V''(r)$, then the set $\{r\in \R\colon\,V''(r) =-2\beta\}$ is finite.
By Lemmas~\ref{lemma:connection_between_second_derivative_and_quotient}--\ref{lemma:strict_inequalities_for_the_second_difference_quotient}, we therefore have
that $\Phi_2 V > -\beta$ and $\Phi_2 \not \ge -M$ for all $M<\beta$.
 So if $\beta
\in (-\infty, \frac12]$, then $t_c = \infty$, while if $\beta \in (\frac12, \infty)$, then $t_c = (\beta-\frac12)^{-1}$ and
$(\mu_{n,t})_{n\in\N}$ is sequentially Gibbs for $t=t_c$ by Corollary~\ref{corollary:overview_Gibbs_non_Gibbs}.
\end{example}

\begin{example}
\label{example:other_potentials_t_c_bigger_zero}
{\bf [Other potentials: $t_c\in (0,\infty]$, sequentially Gibbs at $t=t_c$]}\\
Consider the potential $V(r)= 2\beta(1 + \cos r)$ for some $\beta \in (0,\infty)$.
Then
$V''\ge -2 \beta$ and $V'' \not \ge -M$, and hence $\Phi_2 V \ge -\beta $ and $\Phi_2 V \not \ge -M$ for
$M<\beta$ (see Lemma~\ref{lemma:connection_between_second_derivative_and_quotient}).
So, for $\beta \in (0,\frac12]$ we have $t_c = \infty$, while for $\beta\in (\frac12, \infty)$ we have
$t_c = (\beta - \frac12)^{-1}$ by Corollary~\ref{corollary:overview_Gibbs_non_Gibbs}.
Moreover, if $\beta \in (\frac12,\infty)$, then by Lemma \ref{lemma:strict_inequalities_for_the_second_difference_quotient}
it follows that $\Phi_2 V> -\beta$, and hence $(\mu_{n,t})_{n\in\N}$ is sequentially Gibbs for $t= t_c$.
\end{example}

In the previous two examples the sequence $(\mu_{n,t})_{n\in\N}$ is sequentially Gibbs at $t=t_c$. This
is not always the case, as we show in Example~\ref{example:non_Gibbs_at_t_c} below.

\begin{example} \label{example:non_Gibbs_at_t_c} {\bf [Other potentials: $t_c\in (0,\infty]$, not sequentially Gibbs at $t=t_c$]}\\
Let $g\colon\, \R \rightarrow \R$ be given by
\begin{align}
g(r) =
\begin{cases}
e^{-\frac{1}{|r|-1} + |r|-1} & r\in (-\infty,-1)\cup (1,\infty), \\
0 & r\in [-1,1].
\end{cases}
\end{align}
Because
\begin{flalign}
&& \frac{d}{dr} e^{-\frac{1}{r} + r}
&=  (1+ r^{-2}) e^{-\frac{1}{r} + r} && (r\in (0,\infty)),
\end{flalign}
by L'H\^ opital's rule $\lim_{r\downarrow 0} \frac{d}{dr} e^{-\frac{1}{r} + r} = 0= \lim_{r\uparrow 0} \frac{d}{dr} 0$. Hence $g\in C^1(\R,[0,\infty))$.
Furthermore
\begin{flalign}
&&
\frac{d^2}{dr^2} e^{-\frac{1}{r} + r}
&= r^{-4} (1-2r+2r^2+r^4) e^{-\frac{1}{r} + r} \ge 0
&& (r\in (0,\infty)).
\end{flalign}
So $g$ is a convex function with $\Phi_2 g \ge 0$ (see Lemma \ref{lemma:convexity_of_a_function_in_terms_of_second_difference_quotient})  and
$\Phi_2 g|_{[-1,1]} =0$. Hence $\Phi_2 g \not > 0$.
Note also that $\lim_{r\rightarrow \infty} r^{-2} e^{-\frac{1}{r} + r} = \infty$ by L'H\^opital's rule.
Therefore, for all $\beta \in (0,\infty)$,
\begin{flalign}
\label{eqn:g_minus_polynomial_goes_to_infinity_at_infinity}
\lim_{|r|\rightarrow \infty} g(r) - 2 \beta r^2 = \infty.
\end{flalign}
Let $\beta \in (0,\infty)$ and consider $V\in C^1(\R,[0,\infty))$ given by
\begin{flalign}
&& V(r) = g(r) -  \beta r^2 - C_\beta && (r\in \R),
\end{flalign}
where $C_\beta = \inf_{s\in \R} g(s) - \beta s^2$ (which exists because of \eqref{eqn:g_minus_polynomial_goes_to_infinity_at_infinity}).
By Lemma \ref{lemma:connection_between_second_derivative_and_quotient}, $\Phi_2 V \ge - \beta$ and $\Phi_2 V|_{[-1,1]}=- \beta$ and thus also $\Phi_2 V \not > - \beta$.
So, for $\beta \in (0,\frac12]$ we have $t_c = \infty$, while for $\beta \in (\frac12,\infty)$ we have $t_c = (\beta - \frac12)^{-1}$ and $(\mu_{n,t})_{n\in\N}$ is not sequentially Gibbs for $t = t_c$ by Corollary \ref{corollary:overview_Gibbs_non_Gibbs}.
\end{example}

\begin{example}
\label{example:other}
{\bf [Other potentials: $t_c = 0$]}\\
Consider the potential $V(r)= 1- \cos(r^2)$. Then
\begin{flalign}
&& V' (r) & = 2 r \sin(r^2) && (r\in \R), \\ \nonumber
&& V''(r) &= 2\left[ 2 r^2 \cos(r^2) -\sin(r^2)\right] && (r\in \R), \\ \nonumber
&& V''(\pm \sqrt{\pi k}) &= (-1)^k 4\pi k && (k\in\N).
\end{flalign}
Hence $V'' \not \ge -M$ for all $M\in (0,\infty)$, and hence $(\mu_{n,t})_{n\in\N}$ is sequentially Gibbs
for $t=0$ but not for $t\in (0,\infty)$ (see Corollary \ref{corollary:overview_Gibbs_non_Gibbs}).
\end{example}

We end with an example of a sequence of finite-volume mean-field Gibbs measures that is not sequentially Gibbs.

\begin{example}
\label{example:mfg_not_seq_gibbs}
\textbf{[A sequence of finite-volume mean-field Gibbs measures that is not sequentially Gibbs]} \\
Let $V \in C(\R,[0,\infty))$ be given by $V(r) =  |r|$ for $r\in \R$. 
Then $(\mu_{n,0})_{n\in\N}$ is a sequence of finite-volume mean-field Gibbs measures, but it is not sequentially Gibbs as we will show. Indeed, for all sequences $(\alpha_n)_{n\in\N}$,
\begin{flalign}
& \int_\R \1_A(r) e^{-n V(\frac{n-1}{n} \alpha_n + \frac{r}{n})} e^{-\frac{r^2}{2}} \D r  &&  (A\in \cB(\R)).\\
\notag & = e^{(n-1)\alpha_n} \int_{-\infty}^{-(n-1)\alpha_n}  \1_A(r)  e^{-\frac{r^2}{2}+r} \D r
+ e^{-(n-1)\alpha_n} \int_{-(n-1)\alpha_n}^\infty  \1_A(r)  e^{-\frac{r^2}{2}- r} \D r. \hspace{-3cm} && 
\end{flalign}
If $\alpha_n \ge 0$, then by substitution we get (using $\int_{-\infty}^0 e^{-\frac12 r^2} \D r = \sqrt{\frac{\pi}{2}}$):
\begin{align}
e^{(n-1)\alpha_n} \int_{-\infty}^{-(n-1)\alpha_n}  e^{-\frac{r^2}{2}+r} \D r 
\le \int_{-\infty}^0 e^{-  \frac{(r-(n-1)\alpha_n)^2}{2} + r} \D r 
\le \sqrt{\frac{\pi}{2}} e^{- \frac12 ( (n-1)\alpha_n)^2} .
\end{align}
Hence, if $\alpha_n= (n-1)^{-\frac{1}{2}}$ for $n \ge 2$, then $\alpha_n \downarrow 0$, \  $(n-1)\alpha_n \rightarrow \infty$ and, for $A\in \cB(\R)$,
\begin{align}
0\le \lim_{n\rightarrow \infty} e^{2(n-1)\alpha_n} \int_{-\infty}^{-(n-1)\alpha_n}  \1_A(r)  e^{-\frac{r^2}{2}+r} \D r
 \le \lim_{n\rightarrow\infty} 
 \sqrt{\frac{\pi}{2}} e^{(n-1)\alpha_n  \left(1 -  \frac12 (n-1)\alpha_n\right)}
 =0,
\end{align}
and hence ($\overline \gamma_n$ as in \eqref{eqn:formula_for_overline_gamma_n,t})
\begin{flalign}
\lim_{n\rightarrow\infty} \overline \gamma_n(\alpha_n,A)
&= 
\lim_{n\rightarrow\infty} 
\frac{
e^{2(n-1)\alpha_n} \int_{-\infty}^{-(n-1)\alpha_n}  \1_A(r)  e^{-\frac{r^2}{2}+r} \D r
+ \int_{-(n-1)\alpha_n}^\infty  \1_A(r)  e^{-\frac{r^2}{2}-r} \D r
}{
e^{2(n-1)\alpha_n} \int_{-\infty}^{-(n-1)\alpha_n} e^{-\frac{r^2}{2}+r} \D r
+  \int_{-(n-1)\alpha_n}^\infty  e^{-\frac{r^2}{2}-r} \D r
} \hspace{-3cm} \\
\notag & = \mu_{\cN(-1,1)}(A) && (A\in \cB(\R)). 
\end{flalign}
Similarly, if $\alpha_n= -(n-1)^{-\frac{1}{2}}$ for $n \ge 2$, then $\alpha_n \uparrow 0$, \  $(n-1)\alpha_n \rightarrow - \infty$ and $\overline \gamma_n(\alpha_n,A) \rightarrow \mu_{\cN(1,1)}(A)$ for $A\in \cB(\R)$. 
From this we conclude that $(\mu_{n,0})_{n\in\N}$ is not sequentially Gibbs.
\end{example}


\subsection{Discussion}
\label{S1.8}

{\bf 1.}
If $V$ has a power series expansion $V(x) = \sum_{k\in\N} J_k x^k$, $x \in \R$, then
\begin{align}
\label{Hamiltonianrew}
-n (V \circ m_n) (x_1,\dots,x_n) = - \sum_{k\in\N} \frac{J_k}{n^{k-1}} \sum_{i_1,\ldots,i_k=1}^n
\prod_{j=1}^k x_{i_j},
\end{align}
i.e., the system with $n$ spins has a mean-field $k$-spin interaction of strength $J_k/n^{k-1}$
for $k\in\N$. The special case with $J_k \geq 0$ for all $k\in\N$ is called the ferromagnetic model.

\medskip\noindent
{\bf 2.}
Redig and Wang~\cite{ReWa12} analysed our model for a restricted class of potentials. Short-time
Gibbsianness  (i.e., the time-evolved state is Gibbs up to a strictly positive
time) was proved under the condition that the second derivative of the potential exists and
is bounded from below. Several scenarios of Gibbs-non-Gibbs transitions were discussed. Our paper
considers a very general class of positive potentials and provides the precise connection between
bifurcation of minimising trajectories and loss of Gibbsianness.

\medskip\noindent
{\bf 3.}
Our paper contains the first example of an initial Gibbs state and a stochastic dynamics for which
there is immediate loss of Gibbsianness. For all the models that were considered in the literature
so far, short-time Gibbsianness occurs. See e.g.\ \cite{DeRo05}, \cite{vEFedHoRe02}, \cite{FedHoMa13}, \cite{KuRe06}, \cite{LeNyRe02}.

\medskip\noindent
{\bf 4.}
In case the independent Brownian motions are replaced by independent Ornstein-Uhlenbeck processes,  we get
\begin{align}
r \mapsto V(r) + \frac{r^2}{2} + \frac{(e^t r - \alpha)^2}{e^{2t}-1}
- \inf_{s\in \R} V(s) + \frac{s^2}{2} + \frac{(e^t s - \alpha)^2}{e^{2t}-1}
\end{align}
instead of \eqref{eqn:rate_function_of_eta_n,t} (cf.\ \cite[Eq.\ (25)]{ReWa12}), and so we obtain
completely analogous results (in Corollary~\ref{corollary:overview_Gibbs_non_Gibbs}
the condition $\Phi_2 V > - \frac{1+t}{2t}$ is replaced by $\Phi_2 V > - (e^{2t}-1)^{-1}$). In a
forthcoming paper we will investigate what happens when the independent Brownian motions
are replaced by independent diffusions.

\medskip\noindent
{\bf 5.}
For $n\in\N$ and $t>0$, we can write $\mu_{n,t}$ as (compare with \eqref{eq:mfg})
\begin{flalign}
&& \mu_{n,t}(A)
=  \frac{1}{Z_n} \int_{\R^n} \1_A(x)\, e^{-n(V_{n,t} \,\circ\, m_n)(x)}
 \D \mu_{\cN(0,(1+t)I_n)}(x) && (A\in \cB(\R^n))
\end{flalign}
with
\begin{flalign}
&& V_{n,t}(r)
= - \frac{1}{n} \log\left[ \int_{\R}  e^{-nV(s)}\D \mu_{\cN(\frac{r}{1+t}, \frac{t}{n(1+t)})}(s)  \right]
&& (r\in \R)
\end{flalign}
(see \eqref{eqn:mu_n,0,t} in Appendix~\ref{appendix}). The sequence
\begin{equation}
\left(\mu_{\cN(\frac{r}{1+t}, \frac{t}{n(1+t)})}\right)_{n\in\N}
\end{equation}
satisfies the large deviation principle with rate $n$ and rate function $s\mapsto \frac12 (s-\frac{r}{1+t})^2
\left(\frac{1+t}{t}\right)$. Therefore, by Varadhan's Lemma (see den Hollander~\cite[Theorem III.13]{dH00}),
\begin{flalign}
\label{eqn:limit_V_n,t}
&& \lim_{n\to\infty} V_{n,t}(r)
= \inf_{s\in\R} \left[ V(s) + \tfrac12 \left(s-\frac{r}{1+t}\right)^2 \left(\frac{1+t}{t}\right) \right]
&& (r\in\R).
\end{flalign}
Note that, in the context of Definition~\ref{def:mean-field_Gibbs_sequence}, we are interested in the
behaviour of $\mu_n$ for large $n$ only. Therefore, looking back at Definition~\ref{def:mean-field_Gibbs_sequence},
we may generalise the notion of a sequence of finite-volume mean-field Gibbs measures with potential $V$, namely replacing $V$ in \eqref{eq:mfg} by a sequence of potentials $(V_n)_{n\in\N}$ that converges to $V$ in an appropriate sense. Then $(\mu_{n,t})_{n\in\N}$ becomes
a ``generalised'' sequence of finite-volume mean-field Gibbs measures with (limiting) potential $V_t(r) = \lim_{n\rightarrow \infty} V_{n,t}(r)$ as
given in \eqref{eqn:limit_V_n,t}.
It is then interesting to investigate how the regularity of $V_t$ is related to the  sequentially Gibbs property of the sequence $(\mu_{n,t})_{n\in \mathbb{N} }$ (compared to Theorem \ref{theorem:for_a_continuously_differentiable_potential_with_certain_properties_we_get_a_Gibbs_sequence}(c)).


\section{Proof of Lemma~\ref{lemma:gibbsian_sequences_have_sort_of_specification_kernel}}
\label{s:speckernel}

\begin{lemma}
\label{lemma:convergence_of_functions_for_all_sequences_gives_a_continuous_limiting_function}
Let $(\cX,d_{\cX})$ and $(\cY,d_{\cY})$ be metric spaces. Let $f_n\colon\,\cX \rightarrow \cY$ for
$n\in\N$ and suppose that there exists an $f\colon\, \cX \rightarrow \cY$ such that,
for all $x\in \cX$ and for all sequences
$(x_n)_{n\in\N}$ in $\cX$ with $x_n \rightarrow x$ we have $f_n(x_n) \rightarrow f(x)$.
Then $f$ is continuous.
\end{lemma}

\begin{proof}
The proof is elementary. Let $(x_n)_{n\in\N}$ be a sequence in $\cX$ that converges to an element
$x\in \cX$. We first prove that $f_{k_n}(x_n) \rightarrow f(x)$ for all strictly increasing sequences
$(k_n)_{n\in\N}$ in $\N$. To that end, define the sequence $(y_m)_{m\in\N}$ in $\cX$ by putting
$y_m=x$ for $m\in\N\setminus \{k_n\colon\,n\in\N\}$ and $y_{k_n}= x_n$ for $n\in\N$. Then $y_m
\rightarrow x$, hence $f_m(y_m) \rightarrow f(x)$, in particular, $f_{k_n}(x_n)=f_{k_n}(y_{k_n})
\rightarrow f(x)$. Since $f_k(x_n) \xrightarrow{k\rightarrow \infty} f(x_n)$ for all $n\in\N$, we can find
a strictly increasing sequence $(k_n)_{n\in\N}$ for which $d_{\cY}(f_{k_n}(x_n),f(x_n))<\frac1n$ for
all $n\in\N$. Hence $d_{\cY}(f(x_n),f(x)) \le d_{\cY}(f_{k_n}(x_n),f(x_n)) +d_{\cY}(f_{k_n}(x_n),f(x))
\rightarrow 0$.
\end{proof}

\begin{proof}[Proof of Lemma \ref{lemma:gibbsian_sequences_have_sort_of_specification_kernel}]
The proof of weak continuity of the map $\alpha \mapsto \gamma(\alpha,\cdot)$ is an adaptation of the
proof of Lemma \ref{lemma:convergence_of_functions_for_all_sequences_gives_a_continuous_limiting_function}.
Weak continuity of the map $\alpha \mapsto \gamma(\alpha,\cdot)$ implies continuity of the maps
$\alpha \mapsto \int_\R f(x) \D \, [\gamma(\alpha, \cdot)](x)$ for $f\in C_b(\R)$. For open $A\in \cB(\R)$ there exists
a sequence $(f_n)_{n\in\N}$ in $C_b(\R)$ with $f_n \uparrow \1_A$ (point wise). It follows that
$\int_\R f_n(x) \D \, [\gamma(\alpha,\cdot)](x) \uparrow \gamma(\alpha, A)$ for all $\alpha \in \R$ and open $A$,
and so $\alpha \mapsto \gamma(\alpha,A)$ is measurable for all $A\in \cB(\R)$ (since the open sets
generate the Borel sigma-algebra).
\end{proof}


\section{Proof of Theorem
\ref{theorem:for_a_continuously_differentiable_potential_with_certain_properties_we_get_a_Gibbs_sequence}}
\label{section:proof_theorem_initial_seq_is_Gibbs}

\begin{proof}[Proof of Theorem \ref{theorem:for_a_continuously_differentiable_potential_with_certain_properties_we_get_a_Gibbs_sequence}]
It is not hard to check that $\gamma_n$ is a regular conditional probability under
$\rho_n$ of the first coordinate given the magnetisation of the other coordinates.
To see that $\gamma_n$ is proper and weakly continuous, we refer to Appendix \ref{appendix2}.
Let $(\alpha_n)_{n\in\N}$ be a sequence that converges to $\alpha$.
Let $\delta>0$ be such that $V$ is continuously differentiable on $B(\alpha, 2\delta)$. Then, by the mean value theorem,
\begin{flalign}
&& \lim_{n\rightarrow \infty} \1_{[-n\delta ,n\delta]}(y)  \1_A(y) e^{ - n[V(\alpha_n + \frac{y}{n})- V(\alpha_n)]} = \1_A(y) e^{-y V'(\alpha)} && (y\in \R, A\in \cB(\R)).
\end{flalign}
Let $N\in\N$ be such that $\alpha_n \in B(\alpha,\delta)$ for all $n\ge N$.
Then, by the mean value theorem,
\begin{flalign}
&&  e^{ - n[V(\alpha_n + \frac{y}{n})- V(\alpha_n)]} \le e^{  \sup_{s\in B(\alpha,2\delta) } |V'(s)| |y|} && (y\in [-n\delta ,n \delta], n\ge N).
\end{flalign}
Since $y \mapsto e^{\sup_{s\in B(\alpha, 2\delta ) } |V'(s)| |y|}$ is $\mu_{\cN(0,1)}$-integrable,  Lebesgue's Dominated Convergence Theorem implies
\begin{flalign}
\label{eqn:bound_integral_compact}
\lim_{n\rightarrow \infty} \int_{[-n \delta,n \delta]} 
\1_A(y)\, e^{ - n[V(\alpha_n + \frac{y}{n})- V(\alpha_n)]} e^{-y^2/2}
\D y = \int_\R \1_A(y)\, e^{ - yV'(\alpha)}\, e^{-y^2/2} \D y
&&\big(A\in \cB(\R)\big).
\end{flalign}
Furthermore (because $n\le e^n$, $n^2 = n(n-1) + n$ and $V\ge 0$)
\begin{flalign}
\label{eqn:bound_integral_complement_compact_2}
& \int_{[-n\delta ,n \delta ]^c}  e^{ - n[V(\alpha_n + \frac{y}{n})- V(\alpha_n)]} e^{-y^2/2} \D y
 =  n \int_{[-\delta,\delta]^c} e^{ - n[V(\alpha_n + z)- V(\alpha_n)]} e^{-n^2 z^2/2} \D z   \\
\notag
& \le  n \int_{[-\delta,\delta]^c} e^{ n V(\alpha_n)} e^{-n^2 z^2/2} \D z
\le e^{ - n \left[\frac{n-1}{2} \delta^2 - (V(\alpha_n)+1) \right]}  \int_{[-\delta,\delta]^c} e^{-n z^2/2} \D z,
\end{flalign}
where the last term converges to $0$ as $n\rightarrow \infty$.
So, by \eqref{eqn:bound_integral_compact} -- \eqref{eqn:bound_integral_complement_compact_2},
\begin{flalign}
\int_\R  \1_A (y)\, e^{ -n[V(\alpha_n + \frac{y}{n})- V(\alpha_n)]}\, e^{-y^2/2} \D y \to
\int_\R  \1_A(y)\, e^{ - yV'(\alpha)} e^{-y^2/2} \D y
&& \big(A\in \cB(\R)\big),
\end{flalign}
and hence, by \eqref{eqn:formula_for_overline_gamma_n,t}, $\lim_{n\rightarrow \infty} \overline \gamma_n(\alpha_n,A) = \mu_{\cN(-V'(\alpha),1)}(A)$ for all $A\in \cB(\R)$, i.e., $(\overline \gamma_n(\alpha_n,\cdot))_{n\in\N}$ converges strongly (and hence weakly) to $\mu_{\cN(-V'(\alpha),1)}$.
\end{proof}


\section{Proof of Lemma \ref{lemma:main_lemma_about_equivalences}}
\label{proofs_of_bifurcation}

Section~\ref{ss:twopreplem} contains two preparatory lemmas
(Lemmas~\ref{lemma:inequalities_for_g_n,t_in_terms_of_G}--\ref{lemma:estimate_of_limit_in_integral_by_balls})
that provide estimates on $g_{n,t}$ in \eqref{eqn:definition_g_n,t}. These lemmas will be needed
in Section~\ref{ss:proofequiv} to give the proof.


\subsection{Two preparatory lemmas}
\label{ss:twopreplem}

Define $I_{t,\alpha}\colon\, \R \rightarrow [0,\infty)$ for $t\in (0,\infty)$ and $\alpha \in \R$ by
\begin{flalign}
\label{eqn:I_talpha}
&& I_{t,\alpha}(r)= V(r) + \left(r-\frac{\alpha}{1+t}  \right)^2 \frac{1+t}{2t} && (r\in \R).
\end{flalign}
Note that $r\mapsto I_{t,\alpha}(r) - \inf_{s\in \R} I_{t,\alpha}(s)$ is equal to \eqref{eqn:rate_function_of_eta_n,t}.
Hence (see \eqref{eqn:formula_for_eta_nt})
\begin{flalign}
\label{eqn:formula_for_eta_nt2}
&&\eta_{n,t}(\alpha,A)
& = \frac{ \int_\R \1_A(s)\,  e^{-n I_{t,\alpha}(s) } \D s}{\int_\R  e^{-n I_{t,\alpha}(s) } \D s}
&& \big(\alpha\in\R, A\in \cB(\R), n\in\N, t\in (0,\infty) \big).
\end{flalign}

\begin{lemma}
\label{lemma:inequalities_for_g_n,t_in_terms_of_G}
For every $t \in (0,\infty)$ there exists an $L>0$ such that, for all $n\in\N_{\ge 2}$,
\begin{flalign}
\label{eqn:boundedness_of_g_nt}
&&
&g_{n,t}(\alpha,s) \le L e^{-\frac{\alpha}{t}s+\frac{1}{4}s^2} G_t(n,\alpha)
&& (\alpha, s\in \R),
\end{flalign}
where $G_t\colon\, \N  \times \R \rightarrow \R$ is given by
\begin{flalign}
\label{eqn:formula_for_G}
&& G_t(n,\alpha) = \frac{
\int_\R e^{\left( \frac{1+t}{t} \right)^2 z^2}
e^{-n V(z)}  e^{-(n-1)\left( z- \frac{\alpha}{1+t}\right)^2 \frac{1+t}{2t}} \D z}
{\int_\R e^{-nV(r)}  e^{-(n-1)\left( r- \frac{\alpha}{1+t}\right)^2 \frac{1+t}{2t}} \D r}
&& (n\in\N,\alpha\in \R).
\end{flalign}
Consequently, if for a bounded sequence $(\alpha_n)_{n\in\N}$  in $\R$ the sequence
$(G_t(n,\alpha_n))_{n\in\N}$ is bounded as well, then there exists a $\mu_{\cN(0,1)}$-integrable
function $h\colon\, \R \rightarrow [0,\infty)$ for which, for all $n\in\N$,
\begin{flalign}
\label{eqn:inequalty_bound_for_g_n,t}
&& g_{n,t}(\alpha_n,s) \le h(s)
&& (s\in \R).
\end{flalign}
\end{lemma}

\begin{proof}
After some elementary computations (see \eqref{note:g_n,t_rewritten_for_bound} in
Appendix~\ref{appendix}), we may rewrite \eqref{eqn:definition_g_n,t} as
\begin{flalign}
&& & g_{n,t}(\alpha,s)  = \\
\notag && & \frac{n}{n-1} e^{-\frac{\alpha}{t}s}
 \frac{ \int_\R e^{\left[   - 2z^2  + 2 (s+\frac{\alpha}{1+t} ) z -\frac{1}{n-1}(z-s)^2\right]\frac{1+t}{2t}}
e^{-n V(z)}  e^{-(n-1)\left( z- \frac{\alpha}{1+t}\right)^2 \frac{1+t}{2t}}     \D z }{
\int_\R e^{-nV(r)}  e^{-(n-1)\left( r- \frac{\alpha}{1+t}\right)^2 \frac{1+t}{2t}} \D r }
&& (\alpha,s\in \R).
\end{flalign}
Since $-z^2 +2\frac{\alpha}{1+t}z = - (z-\frac{\alpha}{1+t})^2+(\frac{\alpha}{1+t})^2$ and
$\frac{1+t}{t}sz \le \frac{1}{4} s^2 + (\frac{1+t}{t})^2 z^2$, we get
\begin{flalign}
&& g_{n,t}(\alpha,s)
\le 2 e^{\left(\frac{\alpha}{1+t}\right)^2} e^{-\frac{\alpha}{t}s} e^{\frac{1}{4}s^2}
\frac{ \int_\R e^{\left( \frac{1+t}{t} \right)^2 z^2  }
e^{-n V(z)}  e^{-(n-1)\left( z- \frac{\alpha}{1+t}\right)^2 \frac{1+t}{2t}}     \D z }{
\int_\R e^{-nV(r)}  e^{-(n-1)\left( r- \frac{\alpha}{1+t}\right)^2 \frac{1+t}{2t}} \D r }
&& (\alpha,s\in \R),
\end{flalign}
which yields \eqref{eqn:boundedness_of_g_nt}. The claim in \eqref{eqn:inequalty_bound_for_g_n,t} follows
from \eqref{eqn:boundedness_of_g_nt} because $s\mapsto L e^{l|s|+\frac{1}{4}s^2}$ is
$\mu_{\cN(0,1)}$-integrable for all $l \in \R$.
\end{proof}

\begin{lemma}
\label{lemma:estimate_of_limit_in_integral_by_balls}
Let $V\in C^1(\R,[0,\infty))$ and $t\in (0,\infty)$. For all $q,s,\alpha \in \R$, all sequences
$(\alpha_n)_{n\in\N}$ with $\alpha_n \rightarrow \alpha$ and all $\epsilon>0$, there exist
$\delta>0$, $N\in\N$ and $M>0$ such that for all $n\ge N$,
\begin{align}
\notag & \left| g_{n,t}(\alpha_n,s)
-  e^{-sV'(q)} \right| \vee \left|G_t(n,\alpha_n)- e^{\left( \frac{1+t}{t}\right)^2 q^2}\right| \\
\label{eqn:g_nt_minus_exp_-sV'(q)}
&\qquad \qquad \le\epsilon+M
\frac{ \int_{B(q,\delta)^c}
e^{2\left( \frac{1+t}{t}\right)^2 (r-\alpha_n)^2}
e^{-(n-1)  V(r)\wedge V\left(r+\frac{1}{n}(s-r)\right)}
e^{-(n-1) ( r-\frac{\alpha_n}{1+t} )^2  \frac{(1+t)}{2t} }  \D r
}{
\int_{B(q,\delta)} e^{-V(r)} e^{-(n-1) \left[ V(r) +( r-\frac{\alpha_n}{1+t} )^2  \frac{(1+t)}{2t} \right]}  \D r },
\end{align}
where  $G_t(n,\alpha)$ is as in \eqref{eqn:formula_for_G}.
\end{lemma}

\begin{proof}
Let $q,s,\alpha \in \R$, let $(\alpha_n)_{n\in\N}$ be a sequence in $\R$ with
$\alpha_n \rightarrow \alpha$, and let $\epsilon>0$. Let $\delta>0$ be such that
\begin{flalign}
&& \left| e^{-s V'(q)} - e^{-sV'(r)} \right| <\epsilon,
\qquad \left|e^{\left( \frac{1+t}{t} \right)^2 r^2}- e^{\left( \frac{1+t}{t} \right)^2 q^2}\right|<\epsilon
&& (r \in B(q,2\delta)).
\end{flalign}
Let $N\in\N$ be such that $\frac{|s|+|q|+\delta}{N}<\delta$. By the Mean Value Theorem, we have
\begin{flalign}
&&\sup_{r\in B(q,\delta)} \left| e^{-n \left( V(r+ \frac1n (s-r)) - V(r) \right)} - e^{-sV'(q)}\right| <\epsilon
&&(n \geq N).
\end{flalign}
Hence
\begin{align}
& \int_{\R}  \left| e^{-n\big[V( r+ \frac{1}{n}(s-r))- V(r)\big]} - e^{-sV'(q)} \right| \ e^{-V(r)} e^{- (n-1)
 \big[ V(r) + \left( r-\frac{\alpha_n}{1+t}\right)^2  \frac{(1+t)}{2t} \big]}  \D r  \\ \nonumber
& \le  \epsilon \int_{B(q,\delta)}  e^{- (n-1) \big[ V(r)
+ \left( r-\frac{\alpha_n}{1+t}\right)^2  \frac{(1+t)}{2t} \big]}  \D r  \\ \nonumber
& \quad  + e^{-sV'(q)} \int_{B(q,\delta)^c}  e^{- (n-1) \big[ V(r)
+ \left( r-\frac{\alpha_n}{1+t}\right)^2  \frac{(1+t)}{2t} \big]}  \D r \\ \nonumber
& \quad  + \int_{B(q,\delta)^c} e^{- (n-1) \big[ V(r+\frac{1}{n}(s-r))
+ \left( r-\frac{\alpha_n}{1+t}\right)^2  \frac{(1+t)}{2t}\big]}  \D r
\end{align}
and
\begin{align}
& \int_{\R}  \left|e^{\left( \frac{1+t}{t}\right)^2 r^2}  - e^{\left( \frac{1+t}{t}\right)^2 q^2} \right| \
e^{-V(r)} e^{- (n-1) \big[ V(r) + \left( r-\frac{\alpha_n}{1+t}\right)^2  \frac{(1+t)}{2t} \big]}  \D r  \\ \nonumber
& \le  \epsilon \int_{B(q,\delta)}  e^{- (n-1) \big[ V(r) + \left( r-\frac{\alpha_n}{1+t}\right)^2  \frac{(1+t)}{2t} \big]}  \D r  \\
\nonumber
& \quad  + e^{\left( \frac{1+t}{t}\right)^2 q^2} \int_{B(q,\delta)^c}  e^{- (n-1)
\big[ V(r) + \left( r-\frac{\alpha_n}{1+t}\right)^2  \frac{(1+t)}{2t} \big]}  \D r \\ \nonumber
& \quad  + \int_{B(q,\delta)^c} e^{\left( \frac{1+t}{t}\right)^2 r^2} e^{- (n-1)
\big[ V(r+\frac{1}{n}(s-r))  +\left( r-\frac{\alpha_n}{1+t}\right)^2  \frac{(1+t)}{2t}\big]}  \D r.
\end{align}
Because $e^{-(n-1) V(r)} \vee e^{-(n-1) V(r+ \frac{1}{n}(s-r))} \le  e^{-(n-1) [ V(r)\wedge V(r+\frac{1}{n}(s-r)) ]}$,
we obtain
\begin{align}
\notag
& \left| g_{n,t}(\alpha_n,s)
-  e^{-sV'(q)} \right| \vee \left|G_t(n,\alpha_n)- e^{\left( \frac{1+t}{t}\right)^2 q^2}\right| \\
&\qquad \qquad \le\epsilon+K
\frac{ \int_{B(q,\delta)^c}
e^{\left( \frac{1+t}{t}\right)^2 r^2}
e^{-(n-1) \left[ V(r)\wedge V\left(r+\frac{1}{n}(s-r)\right) + ( r-\frac{\alpha_n}{1+t} )^2  \frac{(1+t)}{2t} \right]}  \D r
}{
 \int_{B(q,\delta)} e^{-V(r)} e^{-(n-1) \left[ V(r) +( r-\frac{\alpha_n}{1+t} )^2  \frac{(1+t)}{2t} \right]}  \D r }
\end{align}
with $K= e^{-sV'(q)} + e^{\left( \frac{1+t}{t}\right)^2 q^2}+ 1$ (see \eqref{eqn:definition_g_n,t*}
in Appendix~\ref{appendix}). Because $r^2 \le 2 \left(r-\frac{\alpha_n}{1+t}\right)^2 + 2\left(\frac{\alpha_n}{1+t}\right)^2$ and $(\alpha_n)_{n\in\N}$
is bounded, we get \eqref{eqn:g_nt_minus_exp_-sV'(q)}.
\end{proof}


\subsection{Proof of Lemma \ref{lemma:main_lemma_about_equivalences}}
\label{ss:proofequiv}

In the proof we use the identity
\begin{align}
\label{eqn:remark:algebraic_computation_with_n_in_square1}
\left( r-\frac{\alpha_n}{1+t} \right)^2 \frac{1+t}{2t}
= \left( r-\frac{\alpha}{1+t} \right)^2 \frac{1+t}{2t}
+ \frac{1}{t} \left(r-\frac{\alpha}{1+t}\right)(\alpha- \alpha_n) +  \frac{ \left( \alpha - \alpha_n \right)^2 }{2t(1+t)},
\end{align}
which implies
\begin{align}
\label{eqn:remark:algebric_computation_with_n_in_square2}
I_{t,\alpha_n}(r) = I_{t,\alpha}(r) + \frac{1}{t} \left(r- \frac{\alpha}{1+t}\right)(\alpha- \alpha_n) +  \frac{ \left( \alpha - \alpha_n \right)^2 }{2t(1+t)}.
\end{align}

\begin{proof}[Proof of Lemma \ref{lemma:main_lemma_about_equivalences}]
Let $s,q\in \R$ be the smallest global minimiser of \eqref{eqn:rate_function_of_eta_n,t}, i.e.,
\begin{align}
q= \inf \left\{r\in \R\colon\, I_{t,\alpha}(r) = \inf_{s\in \R} I_{t,\alpha}(s)\right\}.
\end{align}
(A similar argument works for the largest global minimiser.) By Lemmas
\ref{lemma:inequalities_for_g_n,t_in_terms_of_G}--\ref{lemma:estimate_of_limit_in_integral_by_balls}
it suffices to show that, for all $\delta>0$,
\begin{align}
\label{eqn:convergence_of_integral_outside_devided_by_integral_inside_ball}
\frac{ \int_{B(q,\delta)^c}
e^{2\left( \frac{1+t}{t}\right)^2 \left(r- \frac{\alpha_n}{1+t}\right)^2}
e^{-(n-1) \left[ V(r)\wedge V\left(r+\frac{1}{n}(s-r)\right) + ( r-\frac{\alpha_n}{1+t} )^2  \frac{(1+t)}{2t} \right]}  \D r
}{
 \int_{B(q,\delta)} e^{-V(r)} e^{-(n-1) \left[ V(r) +( r-\frac{\alpha_n}{1+t} )^2  \frac{(1+t)}{2t} \right]}  \D r } \rightarrow 0.
\end{align}
For Part (b) we need to consider a particular sequence $(\alpha_n)_{n\in\N}$ in $\R$ converging to $\alpha$,
while for Part (a) we need to consider all sequences $(\alpha_n)_{n\in\N}$ converging to $\alpha$. In
both cases, for $\delta>0$ we provide a sequence $(c_n)_{n\in\N}$ in $\R$ for which we check the following
three steps, which together yield \eqref{eqn:convergence_of_integral_outside_devided_by_integral_inside_ball}:

\medskip\noindent
{\bf Step 1:} Find $R>0$, $C_1>0$ and $N_1\in \N$ for which
\begin{flalign}
\label{eqn:bound_of_step1}
& \int_{B(0,R)^c}
e^{2\left( \frac{1+t}{t}\right)^2 (r-\alpha_n)^2}
e^{-(n-1) \big[ V(r)\wedge V\left(r+\frac{1}{n}(s-r)\right)
+ ( r-\frac{\alpha_n}{1+t} )^2  \frac{(1+t)}{2t} - c_n \big]}  \D r  \le C_1
&& (n\ge N_1).
\end{flalign}
{\bf Step 2:} Find $C_2>0$ and $N_2\in \N$ for which
\begin{flalign}
\label{eqn:bound_of_step2}
& \int_{B(q,\delta)^c \cap B(0,R)}
e^{2\left( \frac{1+t}{t}\right)^2 (r-\alpha_n)^2}
e^{-(n-1) \big[ V(r)\wedge V\left(r+\frac{1}{n}(s-r)\right)
+ \left( r-\frac{\alpha_n}{1+t}  \right)^2 \frac{1+t}{2t}  -  c_n   \big]}  \D r \le C_2
&& (n \ge N_2).
\end{flalign}
{\bf Step 3:} Find $N_3\in \N$ and a sequence $(\Gamma_n)_{n\in\N}$ with $\Gamma_n \rightarrow \infty$
for which
\begin{flalign}
\label{eqn:bound_of_step3}
\int_{B(q,\delta)} e^{-V(r)} e^{-(n-1) \left[ V(r) +( r-\frac{\alpha_n}{1+t} )^2  \frac{(1+t)}{2t} -c_n\right]}  \D r
 \ge \Gamma_n
 && (n\ge N_3).
\end{flalign}

Abbreviate
\begin{align}
c= I_{t,\alpha}(q) = \inf_{r\in \R} I_{t,\alpha}(r) \in [0,\infty).
\end{align}

\medskip\noindent
$\bullet$ {\bf Step 1 for (a) and (b).}
For all bounded sequences $(\alpha_n)_{n\in\N}$ (in particular those that converge to $\alpha$)
there exists an $R>0$ such that
\begin{flalign}
&& \left( r-  \frac{\alpha_n }{1+t} \right)^2  \frac{1}{2t} > c+ 1 && (r\in B(0,R)^c, n\in\N).
\end{flalign}
Therefore, for all sequences $(c_n)_{n\in\N}$ in $\R$ with $c_n \le c+1$ for all $n\in\N$,
\begin{flalign}
&& V(r)\wedge V\left(r+\frac{1}{n}(s-r)\right) + \left( r-\frac{\alpha_n}{1+t} \right)^2  \frac{1+t}{2t}
& > c_n + \left( r-  \frac{\alpha_n }{1+t} \right)^2  \frac{1}{2}, \hspace{-2cm}
\\
\notag && &  &&  (r\in B(0,R)^c, n\in\N).
\end{flalign}
Let $N_1\in \N$ be such that $N_1-1 > 4(\frac{1+t}{t})^2 +1$. Then
\begin{flalign}
\label{eqn:theorem:estimate_outside_ball_with_radius_R}
&& & \int_{B(0,R)^c}
e^{2\left( \frac{1+t}{t}\right)^2 \left(r- \frac{\alpha_n}{1+t}\right)^2}
e^{-(n-1) \big[ V(r)\wedge V\left(r+\frac{1}{n}(s-r)\right)
+ \left( r-\frac{\alpha_n}{1+t} \right)^2  \frac{ 1+t}{2t} - c_n \big]}  \D r\\
\notag
&& & \quad \le
\int_{\R} e^{\left(4\left( \frac{1+t}{t}\right)^2 - (n-1)\right) \left( r- \frac{\alpha_n }{1+t} \right)^2 \frac12} \D r
\le \int_{\R} e^{- \left( r- \frac{\alpha_n }{1+t} \right)^2  \frac{1}{2}} \D r = \sqrt{2\pi}
&& (n\ge N_1).
\end{flalign}

\medskip\noindent
$\bullet$ {\bf Step 2 for (a).}
Because $\lim_{r\rightarrow \pm \infty} I_{t,\alpha}(r)= \infty$, $I_{t,\alpha}$ is continuous and 
$I_{t,\alpha}$ attains its global minimum at $q$, there exists a $\rho \in (0,\frac15)$ for which
\begin{flalign}
&& I_{t,\alpha}(r) > c + 5\rho && (r\in B(q,\delta)^c).
\end{flalign}
Here, and in Step 3 for (a) below, we pick $c_n= c+ 3\rho$ for $n\in\N$. Note that $c_n \le c +1$
for all $n\in\N$. By \eqref{eqn:remark:algebraic_computation_with_n_in_square1} and the continuity
of $V$ there exists an $N_2\in \N$ such that, for all $n\ge N_2$,
\begin{flalign}
&&V(r)\wedge V\left(r+\frac{1}{n}(s-r)\right)
+ \left( r-\frac{\alpha_n}{1+t}  \right)^2 \frac{1+t}{2t} > I_{t,\alpha}(r) - \rho > c+ 4\rho  \hspace{-3cm}\\
\notag && && \big(r\in B(q,\delta)^c \cap B(0,R)\big).
\end{flalign}
Moreover, there exists an $\Upsilon>0$ such that $e^{2\left( \frac{1+t}{t}\right)^2 \left(r-\frac{\alpha_n}{1+t}\right)^2 } 
\le \Upsilon$ for all $n\in\N$ and all $r\in B(0,R)$. Hence we obtain \eqref{eqn:bound_of_step2} with
$C_2= 2R\Upsilon$ (and $c_n = c+ 3\rho$ for $n\in\N$).

\medskip\noindent
$\bullet$ {\bf Step 2 for (b).}
Here, and in Step 3 for (b) below, we consider $\alpha_n= \alpha - \frac{1}{\sqrt{n}}$ for $n\in\N$, and
\begin{flalign}
&& c_n &= I_{t,\alpha_n}(q) +  \frac{\delta}{\sqrt{n}t}
= I_{t,\alpha}(q) + \frac{1}{t} \left(q-\frac{\alpha}{1+t}\right)(\alpha- \alpha_n) 
+  \frac{ \left( \alpha - \alpha_n \right)^2 }{2t(1+t)}
+  \frac{\delta}{\sqrt{n}t}
&&(n\in\N).
\end{flalign}
Note that $c_n \rightarrow c$, and so there exists an $N_1\in \N$ for which $N_1-1 > 4( \frac{1+t}{t})^2 +1$
and $c_n \le c+ 1 $ for $n\ge N_1$ (and thus \eqref{eqn:theorem:estimate_outside_ball_with_radius_R} holds).
For $r\in B(q,\delta)^c\cap B(0,R)$ we write
\begin{align}
\label{eqn:equation_difference_with_c_n_from_remark}
&V(r)\wedge V\left(r+\frac{1}{n}(s-r)\right) + \left( r-\frac{\alpha_n}{1+t}  \right)^2 \frac{1+t}{2t}  -  c_n \\
\notag
&=  \left(V(r)\wedge V\left(r+\frac{1}{n}(s-r)\right) - V(r)\right) + \left(I_{t,\alpha_n}(r) -  c_n  \right).
\end{align}
For the left part (of the right hand side of \eqref{eqn:equation_difference_with_c_n_from_remark}) we have
\begin{align}
V(r)\wedge V\left(r+\frac{1}{n}(s-r)\right) - V(r) \ge -  \frac{1}{n}\Theta
\end{align}
with $\Theta = (\sup_{u\in B(0,R+|s|)}|V'(u)|)(R+|s|)$. For the right part  first note that, by the
definition of $q$ and the continuity of $I_{t,\alpha}$, there exists a $\rho>0$ such that
\begin{flalign}
&&I_{t,\alpha}(r) > I_{t,\alpha}(q)  +\rho
&& \big(r\in (-\infty,q-\delta)\big).
\end{flalign}
Because $\left(\alpha- \alpha_n\right)\frac{1}{t} = \frac{1}{\sqrt{n}t}$, by  \eqref{eqn:remark:algebric_computation_with_n_in_square2} we have for the right part, for $r\in B(0,R)$,
\begin{flalign}
& I_{t,\alpha_n}(r)-  c_n
= I_{t,\alpha}(r) - I_{t,\alpha}(q) +  \left(r-q \right)\frac{1}{\sqrt{n}t} -  \frac{\delta}{\sqrt{n}t} \\ \nonumber
&\ge
\begin{cases}
\rho - (R+ \delta) \frac{1}{\sqrt{n}t} & r<q-\delta,  \\
0 & r> q+ \delta.
\end{cases}
\end{flalign}
Let $N_2\in \N$ be such that $(R+ \delta) \frac{1}{\sqrt{n}t}<\rho$ for $n\ge N_2$. Then, for
$r\in B(q,\delta)^c\cap B(0, R)$,
\begin{flalign}
&& V(r)\wedge V\left(r+\frac{1}{n}(s-r)\right)
+ \left( r-\frac{\alpha_n}{1+t}  \right)^2 \frac{1+t}{2t}  -  c_n \ge  -  \frac{1}{n} \Theta
&& (n\ge N_2).
\end{flalign}
Moreover, there exists a $\Upsilon>0$ such that $e^{2\left( \frac{1+t}{t}\right)^2 \left(r- \frac{\alpha_n}{1+t}\right)^2}
\le \Upsilon$ for all $n\in\N$ and all $r\in B(0,R)$. Therefore we obtain \eqref{eqn:bound_of_step2}
with $C_2 = 2R \Upsilon e^\Theta$.

\medskip\noindent
$\bullet$ {\bf Step 3 for (a).}
For $r\in A= B(q,\delta)\cap \left\{r\in \R\colon\, I_{t,\alpha}(r)< c+ \rho\right\}$ there exists an $N_3\in \N$
for which $I_{t,\alpha_n}(r) < c+ 2\rho$ for all $n\ge N_3$. Hence
\begin{flalign}
\int_{B(q,\delta)} e^{-V(r)} e^{-(n-1) \left[ V(r) +( r-\frac{\alpha_n}{1+t} )^2  \frac{(1+t)}{2t}
- (c+ 3\rho)\right]}  \D r  \ge  e^{(n-1)\rho} \int_{A}   e^{-V(r)}\D r
&& (n\ge N_3).
\end{flalign}

\medskip\noindent
$\bullet$ {\bf Step 3 for (b).}
There exists a $K>0$ such that, for all
 $n\in\N$ and $r\in B(q,\frac{\delta}{n})$,
\begin{align}
&I_{t,\alpha_n}(r) -  c_n  =I_{t,\alpha}(r) - I_{t,\alpha}(q)
+  \left(r-q \right)\frac{1}{\sqrt{n}t} -  \frac{\delta}{\sqrt{n}t} \\ \nonumber
&< \frac{\delta}{n} \sup_{s\in B(q,\delta)} \left| \Ds I_{t,\alpha}(s) \right|
+  \frac{\delta}{t n\sqrt{n}} -  \frac{\delta}{\sqrt{n}t}\\ \nonumber
&< \frac{1}{n} K -  \frac{\delta}{\sqrt{n}t} = \frac{1}{\sqrt{n}} \left( \frac{K}{\sqrt{n}} - \frac{\delta}{t}\right).
\end{align}
Let $N_3\in \N$ be such that $\frac{K}{\sqrt{n}} < \frac12 \frac{\delta}{t}$ for $n\ge N_3$. Then,
for $r\in B(q,\frac{\delta}{n})$,
\begin{flalign}
\label{eqn:denominator_is_less_then_something}
&& V(r) + \left( r-\frac{\alpha_n}{1+t}  \right)^2 \frac{1+t}{2t}
-  c_n < - \frac12 \frac{\delta}{t} \frac{1}{\sqrt{n}}
&& (n\ge N_3).
\end{flalign}
Let $\kappa>0$ be such that $e^{-V(r)} >\kappa$ for all $r\in B(q,\delta)$. Then
\begin{flalign}
&& &  \int_{B(q,\frac{\delta}{n})}  e^{-V(r)} e^{-(n-1) \left[ V(r)
+( r-\frac{\alpha_n}{1+t} )^2  \frac{(1+t)}{2t}  - c_n\right]}   \D r
 \ge \frac{2\delta}{n} \kappa e^{\left(\sqrt{n}-1\right) \frac12 \frac{\delta}{t}  }
&& (n \ge  N_3).
\end{flalign}
\end{proof}


\section{Tools from convex analysis: proof of Theorem~\ref{theorem:rate_function_has_multiple_global_minima_expressed_in_second_differential_quotient} }
\label{s:multglmin}

In this section we state a definition (Definition~\ref{def:supporting_point}) and several lemmas 
(Lemmas~\ref{lemma:some_properties_second_difference_quotient}--\ref{lemma:two_minimisers_f_plus_beta_polynomial})
that are based on convex analysis, and use these to give the proof of 
Theorem~\ref{theorem:rate_function_has_multiple_global_minima_expressed_in_second_differential_quotient}.
After that we prove the claim made below Corollary~\ref{corollary:overview_Gibbs_non_Gibbs} 
(Lemma~\ref{lemma:connection_between_second_derivative_and_quotient}) and make 
an additional observation (Lemma~\ref{lemma:strict_inequalities_for_the_second_difference_quotient}) 
that can be used to determine whether $(\mu_{n,t})_{n\in\N}$ is sequentially Gibbs at $t=t_c$.

\begin{definition}
\label{def:supporting_point}
Let $f\colon\,\R \rightarrow \R$. Then $a\in \R$ is called a \emph{supporting point} for $f$ if
there exists a linear function $l\colon\, \R \to \R$ with $l(a) = f(a)$ and $l(x) \le f(x)$, $x\in \R$.
\end{definition}

\begin{lemma}
\label{lemma:some_properties_second_difference_quotient}
Let $f\colon\, \R \rightarrow \R$. Then
\begin{enumerate}[label=\emph{(\alph*)}]
\item
for $x,y,z\in \R$ with $x<y<z$:
\begin{align*}
\Phi_2 f(x,y,z) = \frac{f(x)}{(x-y)(x-z)}+ \frac{f(y)}{(y-x)(y-z)} + \frac{f(z)}{(z-x)(z-y)}.
\end{align*}
\item for $a,b,c,d\in \R$ with $a<b<c<d$:
\begin{align*}
(d-a) \Phi_2f(a,b,d) &=(b-a) \Phi_2 f(a,b,c)+ (d-c) \Phi_2 f(b,c,d), \\
(d-a) \Phi_2f(a,c,d) &=(c-a) \Phi_2 f(a,b,c)+ (d-b) \Phi_2 f(b,c,d).
\end{align*}
\item for $g\colon\, \R \rightarrow \R$, $\theta, \kappa\in \R$:
\begin{align*}
\Phi_2 (\theta f + \kappa g) = \theta \Phi_2 f + \kappa \Phi_2 g.
\end{align*}
\item for $g(x) = x^2$, $\Phi_2 g =1$ and $\Phi_2 h =0$ if $h(x) = \alpha x + \beta$ for $\alpha, \beta \in \R$.
\end{enumerate}
\end{lemma}

\begin{proof}
The proof can be done by hand. See also Schikhof~\cite[Lemma 29.2]{Sc84}.
\end{proof}

\begin{lemma}
\label{lemma:some_properties_of_Phi_2}
Let $f\colon\,\R \to \R$ and $y\in \R$. Then the following are equivalent:
\begin{enumerate}[label=\emph{(\alph*)}]
\item
$y$ is a supporting point for $f$,
\item
\hfill
\vspace{-\abovedisplayskip}
\vspace{-\baselineskip} 	
\begin{flalign*}
\frac{f(z) - f(y)}{z-y} \ge \frac{f(y)-f(x)}{y-x}  && (x,z\in \R,\, x<y<z),
\end{flalign*}
\item
$ \Phi_2 f(\cdot,y,\cdot) \ge 0 $.
\end{enumerate}
\end{lemma}
\begin{proof}
Straightforward.
\end{proof}

\begin{lemma}
\label{lemma:convexity_of_a_function_in_terms_of_second_difference_quotient}
A function $f\colon\, \R \rightarrow \R$ is convex if and only if $\Phi_2 f \ge 0$.
Moreover, $f$ is strictly convex if and only if $\Phi_2 f >0$.
\end{lemma}

\begin{proof}
See Schikhof and van Rooij~\cite[Theorem 2.2]{ScvRo82}.
\end{proof}

\begin{lemma}
\label{lemma:lower_semicontinuous_infinite_at_infinity_bfb_fc_attains_its_global_minimum}
Let $f\colon\, \R \to \R$ be lower semicontinuous with $\lim_{|x|\to\infty} f(x) = \infty$. Suppose
that $f$ is bounded from below. Then there exists an $a\in \R$ for which $f(a) = \inf_{x\in \R} f(x)$.
In particular, $a$ is a supporting point for $f$.
\end{lemma}

\begin{proof}
Let $c= \inf_{x\in \R} f(x)$. Define $A_n=\{x\in \R\colon\, f(x) \le c+ \frac{1}{n}\}$, $n\in\N$.
Then $A_n$ is compact and $A_{n+1} \subset A_n$ for all $n\in\N$. Therefore there exists an
$a\in \R$ for which $a\in \bigcap_{n\in\N} A_n $.
\end{proof}

\begin{lemma}
\label{lemma:equivalences_having_minimising_hitting_line_for_lower_semic}
Let $f\colon\,\R \to [0,\infty)$ be lower semicontinuous with $\lim_{|x|\to\infty} f(x) = \infty$.
Then the following are equivalent:
\begin{enumerate}[label=\emph{(\alph*)}]
\item
There exists an $\alpha \in \R$ for which $x\mapsto f(x) - \alpha x$ has multiple global minimisers.
\item
There exists a linear  $l\colon\, \R \rightarrow \R$ for which $\#\{x\in \R\colon\,
l(x) = f(x) \} \ge 2$ and $l \le f$.
\item
There exist $a,b,c\in \R$ with $a<b<c$ and $\Phi_2 f(\cdot, a, \cdot) \ge 0$, $\Phi_2 f(\cdot, c, \cdot)
\ge 0$, $\Phi_2 f(\cdot, b, \cdot) \not > 0$.
\item
There exist $a,x,b,y,c\in \R$ with $a\le x<b<y\le c$ and $\Phi_2 f(\cdot, a, \cdot) \ge 0$,
$\Phi_2 f(\cdot, c, \cdot) \ge 0$, $\Phi_2 f(x, b, y) \le 0$.
\end{enumerate}
\end{lemma}

\begin{proof}
The equivalence (a) $\iff$ (b) and the implication (d) $\Rightarrow$ (c) are trivial. \\
(c) $\Rightarrow$ (d). Assume (c). Then there exist $x,y\in \R$ with $x<b<y$ for which $\Phi_2 f(x,b,y)\le 0$.
If $x<a$ and/or $y>c$, then $\Phi_2 f(a,b,y) \le 0$ and/or $\Phi_2(x,b,c)\le 0$ by  Lemma \ref{lemma:some_properties_second_difference_quotient}(b). Therefore we may assume that $x\ge a$ and 
$y \le c$, i.e., we obtain (d). \\
(b) $\Rightarrow$ (c). Assume (b). Let $a,c \in \{x\in \R\colon\, l(x) =f(x)\}$ with $a<c$. Let $b\in (a,c)$. Then
\begin{align}
\frac{f(c) - f(a)}{c-a} =  \frac{l(c) - l(a)}{c-a} =  \frac{l(b) - l(a)}{b-a}  \le \frac{f(b)-f(a)}{b-a},
\end{align}
i.e., $\Phi_2f(a,b,c) \le 0$. \\
(d) $\Rightarrow $ (b). Define $w,z\in \R$ by
\begin{align}
w& = \sup\{s\le b\colon\, \Phi_2 f( \cdot, s, \cdot) \ge 0\}, \\ \nonumber
z& = \inf \{s \ge b\colon\, \Phi_2 f( \cdot, s, \cdot) \ge 0\}.
\end{align}
Because $f$ is lower semicontinuous, we have $\liminf_{s\uparrow w} f(s)\ge f(w)$. Hence, by 
Lemma~\ref{lemma:some_properties_second_difference_quotient}(a), we have, for $q,r\in \R$ with $q<w<r$,
\begin{align}
0 &\le \limsup_{s\uparrow w} \Phi_2 f ( q,s,r) \\ \nonumber
& = \frac{f(q)}{(q-w)(q-r)}+ \frac{f(r)}{(r-w)(r-q)} - \frac{\liminf_{s\uparrow w} f(s)}{(r-w)(w-q)}
 \le \Phi_2 f(q,w,r).
\end{align}
So $\Phi_2 f(\cdot,w,\cdot)\ge 0$. Similarly $\Phi_2 f(\cdot, z, \cdot) \ge 0 $. If $w=b$, then $z=b$, 
and vice versa. \\
$\bullet$ Assume that $w=b=z$. Then $f$ is convex and $\Phi_2 f (x,b,y) =0$. With
$l\colon\, \R \rightarrow \R$, $s\mapsto f(x) + \frac{f(y)-f(x)}{y-x}(s-x)$ one then has $l\le f$
and $l(s)=f(s)$ for all $s\in [x,y]$, since
\begin{align}
\frac{f(b)-f(x)}{b-x} \le \frac{f(s) - f(b)}{s-b} \le \frac{f(y) -f(b)}{y-b}= \frac{f(b)-f(x)}{b-x} .
\end{align}
$\bullet$ Assume that $w<b<z$. Define $l\colon\, \R \rightarrow \R$, $s \mapsto f(w) + \frac{f(z)-f(w)}{z-w}(s-w)$.
Then $l \le f$ on $(w,z)^c$. Note that $f-l|_{[w,z]}$ is lower semicontinuous and bounded from below.
By Lemma \ref{lemma:lower_semicontinuous_infinite_at_infinity_bfb_fc_attains_its_global_minimum},
it attains its infimum at some $a\in [w,z]$. This $a$ is a supporting point of $f$, and hence $a=w$ or 
$a=z$ by Lemma \ref{lemma:some_properties_of_Phi_2}. Thus $l(s) \le f(s)$ for all $s\in \R$.
\end{proof}

\begin{lemma}
\label{lemma:adding_squares_to_a_lsc_function_assures_enough_supporting_points}
Let $f\colon\, \R \to[0,\infty)$ be lower semicontinuous. Let $r\in \R$ and $\beta>0$.
Then there exist $q,s\in \R$ with $q<r<s$ that are supporting points of $x\mapsto f(x)
+\beta x^2$, i.e., $\Phi_2 f (\cdot, q,\cdot) \ge - \beta$, $\Phi_2 f (\cdot,s,\cdot) \ge -\beta$.
\end{lemma}

\begin{proof}
Since $x\mapsto f(x) + \beta x^2$ is lower semicontinuous and $\lim_{|x|\to\infty}
[f(x) + \beta x^2] = \infty$, by 
Lemma \ref{lemma:lower_semicontinuous_infinite_at_infinity_bfb_fc_attains_its_global_minimum}
there exists an $a\in \R$ for which $a$ is a global minimum and thus a supporting point
for $x\mapsto f(x) + \beta x^2$. There exists a (large enough) $\theta>0$ such that
\begin{align}
\label{eqn:theorem_finding_supporting_point_set}
\{x\in \R\colon\, f(a) -1 + \theta (x-r) = \beta x^2 \}
\end{align}
has two elements, say $x_1,x_2$ with $x_1<x_2$. By the definition of $a$, we have $x_1>r$.
By Lemma~\ref{lemma:lower_semicontinuous_infinite_at_infinity_bfb_fc_attains_its_global_minimum},
there exists an $s\in \R$ that is a global minimum and a supporting point of
\begin{align}
\label{eqn:theorem_finding_supporting_point_function_minus_line}
x\mapsto f(x) + \beta x^2 - (f(a) -1 + \theta (x-r)).
\end{align}
Hence $s$ is also a supporting point of $x\mapsto f(x) + \beta x^2$. Because
\eqref{eqn:theorem_finding_supporting_point_function_minus_line} is strictly negative
on $(x_1,x_2)$ and non-negative on $[x_1,x_2]^c$, we have $s\in [x_1,x_2]$. Therefore $s>r$.
There also exists a (small enough) $\theta<0$ for which \eqref{eqn:theorem_finding_supporting_point_set}
has two elements. In the same way we can prove that there is an $q<r$ that is also a supporting
point of $x\mapsto f(x) + \beta x^2$. The last part of the statement is a consequence of 
Lemma~\ref{lemma:some_properties_second_difference_quotient}.
\end{proof}

\begin{lemma}
\label{lemma:two_minimisers_f_plus_beta_polynomial}
Let $f\colon\,\R \to [0,\infty)$ be lower semicontinuous and let $\beta\in (0,\infty)$.
Then there exists an $\alpha \in \R$ for which $x \mapsto f(x) + \beta  x^ 2 - \alpha x $
has multiple global minimisers if and only if $\Phi_2 f \not > - \beta$.
\end{lemma}

\begin{proof}
This is a consequence of 
Lemmas~\ref{lemma:equivalences_having_minimising_hitting_line_for_lower_semic}--\ref{lemma:adding_squares_to_a_lsc_function_assures_enough_supporting_points}.
\end{proof}

\begin{proof}[Proof of Theorem \ref{theorem:rate_function_has_multiple_global_minima_expressed_in_second_differential_quotient}]
The claim in
Theorem~\ref{theorem:rate_function_has_multiple_global_minima_expressed_in_second_differential_quotient}
follows by applying Lemma~\ref{lemma:two_minimisers_f_plus_beta_polynomial} with $\beta= \frac{1+t}{2t}$ to 
the lower semicontinuous function $r\mapsto V(r) + \frac{1}{2}r^2$.
\end{proof}

The following observation proves the claim made below Corollary~\ref{corollary:overview_Gibbs_non_Gibbs}.

\begin{lemma}
\label{lemma:connection_between_second_derivative_and_quotient}
Let $f\colon\,\R \to \R$ be twice differentiable.  Then $f'' \ge 2\beta$ if and only if $\Phi_2 f \ge \beta$ for
all $\beta\in \R$.
\end{lemma}

\begin{proof}
By Lemma~\ref{lemma:convexity_of_a_function_in_terms_of_second_difference_quotient},
$\Phi_2 g \ge 0$ if and only if $g$ is convex. Since a twice differentiable function $g$ is convex if and only 
if $g'' \ge 0$, this implies the equivalence $\Phi_2 f \ge 0 \iff f''\ge 0$. Let $\beta \in \R$ and let $g\colon\,
\R \to \R$ be given by $g(r) = f(r) - \beta r^2$. Then, by
Lemma~\ref{lemma:some_properties_second_difference_quotient}, we have $f'' \ge 2\beta$ $\iff$ $g'' \ge 0$ 
$\iff$ $\Phi_2 g \ge 0$ $\iff$ $\Phi_2 f \ge \beta$.
\end{proof}

In contrast to Lemma \ref{lemma:connection_between_second_derivative_and_quotient}, we can have 
$\Phi_2 f > \beta$ but not $f'' > 2\beta$ (take e.g.\ $\beta =0$ and $f(x) = x^4$, in which case $\Phi_2 f >0$ 
by Lemma \ref{lemma:convexity_of_a_function_in_terms_of_second_difference_quotient} but $f''(0)=0$).
However, according to the next observation the second derivative of $f$ can be used to determine whether 
$\Phi_2 f> \beta$. This observation can be used to determine whether $(\mu_{n,t})_{n\in\N}$ is sequentially 
Gibbs at $t= t_c$.

\begin{lemma}
\label{lemma:strict_inequalities_for_the_second_difference_quotient}
Let $f\colon\, \R \rightarrow \R$. Let $a,b,c\in \R$ with $a<b<c$, and $\beta \in \R$.
\begin{enumerate}[label=\emph{(\alph*)}]
\item
If $\Phi_2 f|_{(a,b)}>\beta$, $\Phi_2 f|_{(a,b]}\ge\beta$, $\Phi_2 f|_{(b,c)}>\beta$, $\Phi_2 f|_{[b,c)}
\ge \beta$ and $\Phi_2 f|_{(a,c)}(\cdot,b,\cdot)\ge 0$, then $\Phi_2 f|_{(a,c)}>\beta$.
\item
If $f$ is upper semicontinuous and $\Phi_2 f|_{(a,b)}\ge\beta$, then $\Phi_2 f|_{[a,b]}\ge\beta$.
\item
If $f$ is twice differentiable on $(a,b)$ and ${f|_{(a,b)}}''>\beta$, then $\Phi_2 f|_{(a,b)}>\beta$.
\end{enumerate}
\end{lemma}

\begin{proof}
Without loss of generality we may assume $b=0$. \\
(a) Let $x,y,z\in (a,c)$. If $x<y<0<z$ or $x<0<y<z$, then with Lemma \ref{lemma:some_properties_second_difference_quotient}(b) we easily get $\Phi_2 f(x,y,z)>\beta$.
If $y=0$, then $x<\frac{x}{2}<0<z$, and hence $\Phi_2 (x,\frac{x}{2},0)>0$. Again with Lemma \ref{lemma:some_properties_second_difference_quotient}(b), we get $\Phi_2 (x,0,z)>0$. \\
(b) If $f$ is upper semicontinuous, then $\limsup_{s\uparrow b} f(s) \le f(b)$ and $\limsup_{s\downarrow a}
f(s) \le f(a)$. Together with Lemma \ref{lemma:some_properties_second_difference_quotient}(a) this 
proves the second statement. \\
(c) If ${f|_{(a,b)}}''>0$, then $f$ is strictly convex, and with Lemma \ref{lemma:convexity_of_a_function_in_terms_of_second_difference_quotient} this implies (c) in
case $\beta =0$. Replacing $f$ by $g(r) = f(r) - \frac{\beta}{2} r^2$, we obtain (c) for $\beta \ne 0$
(see Lemma~\ref{lemma:some_properties_second_difference_quotient}).
\end{proof}


\appendix

\setcounter{note}{0}
\renewcommand{\thenote}{\Alph{section}\arabic{note}}


\section{Key formulas}
\label{appendix}

In this appendix we derive a few formulas that were used in the main body of the paper.

\begin{note}
We derive formulas for $\gamma_{n,t}$ and $\overline \gamma_{n,t}$ described in Section 1.4. \\
Inserting \eqref{eqn:formula_for_P_n} into \eqref{eqn:formula_for_mu_n,t} we get, for $A\in \cB(\R^n)$,
\begin{align}
\label{eqn:formula_for_mu_n,t_rewritten}
 \mu_{n,t}(A)
& =  \frac{1}{Z_n} \int_{\R^n}
\Big[ (2\pi t)^{- \frac{n}{2}} \int_{\R^n} \1_A(y) e^{-\frac{\|y-z\|^2}{2t}}  \D y \Big]
e^{-n (V\,\circ\,m_n) (z)} \D\mu_{\cN(0,I_n)}(z)\\
\notag
& =  \frac{1}{Z_n} \int_{\R^n} \1_A(y) \Big[ (2\pi)^{-n}  t^{- \frac{n}{2}} \int_{\R^n}
e^{-\frac{\|y-z\|^2}{2t}} e^{-\frac{\|z\|^2}{2}}  e^{-n (V\,\circ\,m_n) (z)}  \D z \Big] \D y.
\end{align}
Since $\frac{\|y-z\|^2}{2t}+\frac{\|z\|^2}{2} = \frac{\|y\|^2}{2(1+t)} +\frac{\|\frac{y}{1+t} -z\|^2 (1+t)}{2t}$
for $y,z\in \R^n$, we get, for $A\in \cB(\R^n)$,
\begin{flalign}
\label{eqn:formular_for_mu_n,t_rewritten2}
\mu_{n,t}(A)
=  \frac{1}{Z_n} \int_{\R^n} \1_A(y) \Big[ (2\pi)^{-n}  t^{- \frac{n}{2}} \int_{\R^n}
e^{-\frac{\|y\|^2}{2(1+t)}} e^{-\frac{\|\frac{y}{1+t} -z\|^2 (1+t)}{2t}}  e^{-n (V\,\circ\,m_n) (z)}  \D z \Big] \D y.
\end{flalign}
Then it is not hard to check that $\gamma_{n,t}: \R^{n-1} \times \cB(\R)$
defined for $y_2,\dots,y_n\in \R$ and $B\in \cB(\R)$ by
\begin{align}
\label{eqn:formula_for_gamma_n,t}
 & \gamma_{n,t}((y_2,\dots,y_n),B) \\
\notag & =\frac{ (2\pi(1+t))^{-\frac{n}{2}} \int_{\R} \1_B(x)\, e^{-\frac{x^2}{2(1+t)}}
\int_{\R^n}  e^{-n (V\,\circ\,m_n) (z)}  \D \mu_{\cN(\frac{(x,y_2,\dots,y_n)}{1+t}, \frac{t}{1+t} I_n)}(z) \D x
}{
(2\pi(1+t))^{-\frac{n}{2}} \int_{\R}   e^{-\frac{x^2}{2(1+t)}}
\int_{\R^n}  e^{-n (V\,\circ\,m_n) (z)}  \D \mu_{\cN(\frac{(x,y_2,\dots,y_n)}{1+t}, \frac{t}{1+t} I_n)}(z)
\D x}.
\end{align}
is the weakly continuous proper conditional probability under $\mu_{n,t}$ of the first spin given
the other spins. Using the identities
\begin{align}
\mu_{\cN\left(\frac{(x,y_2,\dots,y_n)}{1+t}, \frac{t}{1+t} I_n\right)}
& = \mu_{\cN(\frac{x}{1+t}, \frac{t}{1+t} )} \otimes
\mu_{\cN\left(\frac{(y_2,\dots,y_n)}{1+t}, \frac{t}{1+t} I_{n-1}\right)}, \\
\mu_{\cN\left(\frac{(y_2,\dots,y_n)}{1+t}, \frac{t}{1+t} I_{n-1}\right)}
\circ m_{n-1}^{-1}
& = \mu_{\cN\left(\frac{m_{n-1}(y_2,\dots,y_n)}{1+t}, \frac{t}{(n-1)(1+t)} \right)}, \\
m_n(z_1,\dots,z_n)
& = \frac{z_1}{n} + \frac{n-1}{n} m_{n-1}(z_2,\dots,z_n),
\end{align}
we obtain the  expression
\begin{align}
\label{eqn:formula_for_gamma_n2}
& \gamma_{n,t}((y_2,\dots,y_n),B) = \\
\notag
& \frac{\int_{\R} \1_B(x)
\int_{\R} \int_{\R}  e^{-nV(\frac{1}{n}s+ \frac{n-1}{n} r)}
\D \mu_{\cN \left(\frac{m_{n-1}(y_2,\dots,y_n)}{1+t}, \frac{t}{(n-1) 1+t} \right)}(r)
\D \mu_{\cN \left(\frac{x}{1+t} , \frac{t}{1+t}\right)}(s) \D \mu_{\cN(0,1+t)}(x)}
{\int_{\R}
\int_{\R} \int_{\R}  e^{-nV(\frac{1}{n}s+ \frac{n-1}{n} r)}
\D \mu_{\cN \left(\frac{m_{n-1}(y_2,\dots,y_n)}{1+t}, \frac{t}{(n-1) 1+t} \right)}(r)
\D \mu_{\cN \left(\frac{x}{1+t} , \frac{t}{1+t}\right)}(s) \D \mu_{\cN(0,1+t)}(x) }.
\end{align}
We see that $\gamma_{n,t}(u,\cdot)= \gamma_{n,t}(v,\cdot)$ for all $u,v\in \R^{n-1}$ with
$m_{n-1}(v)= m_{n-1}(u)$. Hence we can define $\overline \gamma_{n,t}\colon\, \R \times \cB(\R)
\rightarrow [0,1]$ by letting $\overline \gamma_{n,t}(\alpha, B) = \gamma_{n,t}(v,B)$ for
$\alpha \in \R$ and $B\in \cB(\R)$, where $v\in \R^{n-1}$ is such that $m_{n-1}(v) = \alpha$, i.e.,
\begin{flalign}
\label{eqn:definition_of_overline_gamma_n,t}
&& & \overline \gamma_{n,t}(\alpha, B) \\
\notag
&& & = \frac{\int_{\R} \1_B(x)
\int_{\R} \int_{\R}  e^{-nV(\frac{1}{n}s+ \frac{n-1}{n} r)}
\D \mu_{\cN \left(\frac{\alpha}{1+t}, \frac{t}{(n-1) 1+t} \right)}(r)
\D \mu_{\cN \left(\frac{x}{1+t} , \frac{t}{1+t}\right)}(s) \D \mu_{\cN(0,1+t)}(x)}
{\int_{\R}
\int_{\R} \int_{\R}  e^{-nV(\frac{1}{n}s+ \frac{n-1}{n} r)}
\D \mu_{\cN \left(\frac{\alpha}{1+t}, \frac{t}{(n-1) 1+t} \right)}(r)
\D \mu_{\cN \left(\frac{x}{1+t} , \frac{t}{1+t}\right)}(s) \D \mu_{\cN(0,1+t)}(x) }.
\end{flalign}
\end{note}

\begin{note}
\label{note:eta_nt}
We show that $\eta_{n,t}$ is indeed the weakly continuous proper regular conditional probability of the magnetisation of the $n$ spins at time $0$ given the magnetisation at time $t$. \\
Let $\mu_n$ be the law on $C([0,\infty),\R^n)$ of the paths of the independent Brownian motions
performed by the $n$ spins with initial distribution $\mu_{n,0}$,  i.e., $\mu_n$ is given by \eqref{eqn:formula_for_mu_n}.
The joint law of the process at time $0$ and time $t$ is given by
\begin{flalign}
&& \mu_{n,(0,t)}(A) &= \int_\R \int_\R \1_A(x,y) \D\, \left[p_n(t,x,\cdot)\right](y) \D \mu_{n,0}(x)
&& (A\in \cB((\R^2)^n)).
\end{flalign}
We write $m_n$ also for the function $(\R^2)^n \rightarrow \R^2$ given by
\begin{flalign}
&& m_n((x_1,y_1),\dots,(x_n,y_n)) = \frac1n \sum_{i=1}^n (x_i,y_i)
&& (x_1,y_1,\dots,x_n,y_n \in \R).
\end{flalign}
Let $\overline{\mu}_{n,(0,t)} = \mu_{n,(0,t)} \circ m_n^{-1}$. Since $p_n(t,x,\cdot) \circ m_n^{-1}
= \mu_{\cN(x,tI_n)} \circ m_n^{-1} = \mu_{\cN(m_n(x),\frac{t}{n})}$ and $\mu_{\cN(0,I_n)} \circ
m_n^{-1}= \mu_{\cN(0,\frac{1}{n})}$, we have
\begin{flalign}
\label{eqn:mu_n,0,t}
&& \overline{\mu}_{n,(0,t)}(A)
&= \int_\R \int_\R \1_A(s,\alpha) \,\D \mu_{\cN(s,\frac{t}{n})}(\alpha) e^{-nV(s)} \D \mu_{\cN(0,\frac1n)}(s)  \\
\notag
&&
&= \frac{1}{\sqrt{2\pi \frac{t}{n}}} \frac{1}{\sqrt{2\pi \frac{1}{n}}} \int_\R \int_\R \1_A(s,\alpha)\,
e^{-n [V(s) + \frac{s^2}{2} + \frac{(s-\alpha)^2}{2t} ]}   \D s \D \alpha
&& (A\in \cB(\R^2)).
\end{flalign}
From this it follows that $\eta_{n,t}$ given in \eqref{eqn:formula_for_eta_nt} is the weakly continuous
proper regular conditional probability under $\overline{\mu}_{n,(0,t)}$ of the first coordinate given the
second, i.e., the weakly continuous proper regular conditional probability of the magnetisation of the
$n$ spins at time $0$ given the magnetisation at time $t$.
\end{note}

\begin{note}
\label{note:2of_gamma_nt}
We verify \eqref{eqn:definition_g_n,t} and \eqref{eqn:overline_gamma_n,t_in_terms_of_g_nt}.  \\
An elementary computation gives that, for $\alpha,s\in \R$, $t\in (0,\infty)$ and $n\in\N$,
\begin{align}
&\int_{\R}  e^{-nV(\frac{1}{n}s+ \frac{n-1}{n} r)}
\D \mu_{\cN \left(\frac{\alpha}{1+t} , \frac{t}{(n-1) 1+t} \right)}(r) \\ \nonumber
&= \sqrt{\frac{(n-1)(1+t)}{2\pi t}} \int_{\R}  e^{-nV \left(\frac{1}{n}s+ \frac{n-1}{n} r \right)}
e^{- \left(r-\frac{\alpha}{1+t} \right)^2 \frac{(n-1)(1+t)}{2t}}  \D r \\ \nonumber
&= \sqrt{\frac{(n-1)(1+t)}{2\pi t}} e^{-(n-1)\frac{\alpha^2}{1+t}}
\int_{\R}  e^{-n\left[V(r + \frac{1}{n}(s-r))- V(r)\right]}\,
e^{-V(r)}\, e^{-(n-1)[ V(r) + \frac{r^2}{2} + \frac{(r-\alpha)^2}{2t}] }   \D r.
\end{align}
Hence, for $n\in\N$ and $t\in (0,\infty)$, we can write
\begin{flalign}
&& \overline \gamma_{n,t}(\alpha, B)=  \frac{\int_{\R} \1_B(x)
\int_{\R}
g_{n,t}(\alpha,s)
\D \mu_{\cN \left(\frac{x}{1+t} , \frac{t}{1+t} \right)}(s)
\D \mu_{\cN(0,1+t)}(x)}
{\int_{\R} \int_\R
g_{n,t}(\alpha,s)
 \D \mu_{\cN \left(\frac{x}{1+t} , \frac{t}{1+t}\right)}(s) \D \mu_{\cN(0,1+t)}(x) }
&& (\alpha \in \R, B \in \cB(\R)),
\end{flalign}
where
$g_{n,t}\colon\,\R^2 \rightarrow \R$ is as in \eqref{eqn:definition_g_n,t}. With Fubini's
Theorem we have
\begin{flalign}
\label{eqn:rewritting_numerator_of_overline_gamma_nt}
&& & \int_{\R} \1_B(x)
\int_{\R}
g_{n,t}(\alpha,s)
\D \mu_{\cN \left(\frac{x}{1+t} , \frac{t}{1+t}\right)}(s)
\D \mu_{\cN(0,1+t)}(x) &&\\
\notag
&& &=
\frac{1}{2\pi \sqrt{t}} \int_{\R} \int_{\R} \1_B(x)\, g_{n,t}(\alpha,s)
e^{-\left(s-\frac{x}{1+t} \right)^2\frac{1+t}{2t}}\, e^{-x^2\frac{1}{2(1+t)}} \D x \D s &&\\
\notag
&& &=
\frac{1}{2\pi \sqrt{t}}
\int_{\R} \left( \int_{\R} \1_B(x)\, e^{2xs\frac{1}{2t}}\, e^{-x^2\frac{1}{2t}} \D x \right)\,
g_{n,t}(\alpha,s)\, e^{-s^2 \frac{1+t}{2t}}
\D s && \\
\notag
&& &=
\frac{1}{\sqrt{2\pi}} \int_{\R} \mu_{\cN(s,t)}(B)\, e^{s^2\frac{1}{2t}}\,
g_{n,t}(\alpha,s)\, e^{-s^2 \frac{1+t}{2t}} \D s
&& \\
\notag
&& &=
\int_{\R} \mu_{\cN(s,t)}(B)\, g_{n,t}(\alpha,s) \D \mu_{\cN(0,1)}(s)
&& (B\in \cB(\R)).
\end{flalign}
With this we obtain \eqref{eqn:overline_gamma_n,t_in_terms_of_g_nt}.
\end{note}

\begin{note}
\label{note:g_n,t_rewritten_for_bound}
Let $n\in \N_{\ge 2}$ and $t\in (0,\infty)$. Note that, with \eqref{eqn:I_talpha} and \eqref{eqn:formula_for_eta_nt2},
$g_{n,t}$ is given by
\begin{flalign}
\label{eqn:definition_g_n,t*}
&& g_{n,t}(\alpha,s)
& = \frac{ \int_\R e^{-n[ V(r+ \frac1n (s-r))]}\,
e^{-(n-1) \left(r-\frac{\alpha}{1+t}  \right)^2 \frac{1+t}{2t} } \D r }{
\int_\R e^{-nV(r)} e^{-(n-1) \left(r-\frac{\alpha}{1+t}  \right)^2 \frac{1+t}{2t} } \D r }
&& (\alpha,s\in \R).
\end{flalign}
The numerator equals
\begin{align}
\int_\R e^{-n\left[ V(r+\frac{1}{n} (s-r)) \right]} \,
e^{-(n-1)\left( r- \frac{\alpha}{1+t}\right)^2 \frac{1+t}{2t}} \D r
= \frac{n}{n-1} \int_\R e^{-n V(z) }\, e^{-(n-1)\left( \frac{n}{n-1}z - \frac{1}{n-1}s
- \frac{\alpha}{1+t}\right)^2 \frac{1+t}{2t}} \D z.
\end{align}
Via the identities
\begin{align}
& -(n-1)\left( \frac{n}{n-1}z - \frac{1}{n-1}s - \frac{\alpha}{1+t}\right)^2  \\ \nonumber
& = -(n-1) \left( z + \frac{1}{n-1}(z-s) - \frac{\alpha}{1+t}\right)^2  \\ \nonumber
& = -(n-1) \left( z- \frac{\alpha}{1+t}\right)^2
- 2 \left( z- \frac{\alpha}{1+t} \right) (z-s) - \frac{1}{n-1} (z-s)^2   \\ \nonumber
& = -(n-1) \left( z- \frac{\alpha}{1+t}\right)^2
- 2 z^ 2 + 2\left(s+\frac{\alpha}{1+t}\right) z - 2 \frac{\alpha}{1+t}s - \frac{1}{n-1} (z-s)^2,
\end{align}
we get
\begin{flalign}
&& & g_{n,t}(\alpha,s)  = \\ \nonumber
&& & \frac{n}{n-1}\, e^{-\frac{\alpha}{t}s}\,
\frac{
\int_\R e^{\left[-2z^2+2(s+\frac{\alpha}{1+t})z-\frac{1}{n-1}(z-s)^2\right]\frac{1+t}{2t}}\,
e^{-n V(z)}\,e^{-(n-1)(z- \frac{\alpha}{1+t})^2 \frac{1+t}{2t}}\D z
}
{
\int_\R e^{-nV(r)}\, e^{-(n-1)\left(r-\frac{\alpha}{1+t}\right)^2 \frac{1+t}{2t} } \D r
}
&& (\alpha,s\in \R).
\end{flalign}
\end{note}

\begin{note}
We give the proof of Theorem \ref{theorem:rate_function_has_multiple_global_minima_expressed_in_second_differential_quotient}, namely we prove (a), i.e., the existence of $\rho_n$ mentioned in 
Theorem \ref{theorem:rate_path_conditional_seq_and_comparing_with_two_layer} and prove (b), i.e., that 
for $\alpha \in\R$ the large deviation principle holds for $(\rho_n(\alpha,\cdot))_{n\in\N}$ with rate $n$
and rate function given in \eqref{eqn:rate_function_paths}. \\
$\bullet$ \emph{Proof of (a), existence of $\rho_n$}. \\
Let $\fJ=\{(t_0,t_1,\cdots,t_k,t)\colon\, k\in \N_0, 0=t_0<t_1<\cdots<t_k<t\}$ and let $j\in \fJ$ be given by $j=(t_0,t_1,\dots,t_k,t)$. 
Define $\pi_{j}\colon\, C([0,t],\R) 
\rightarrow \R^{k+2}$  by
\begin{flalign}
&& \pi_{j}(\phi) = \big(\phi(t_0), \phi(t_1),\dots,\phi(t_k),\phi(t)\big)
&& \big(\phi\in C([0,t],\R)\big).
\end{flalign}
Similarly as in item \ref{note:eta_nt},  
$\overline \mu_{n,j} := \mu_n \circ \pi_{[0,t]}^{-1} \circ \pi_j^{-1} \circ m_n^{-1}
= (\mu_n \circ \pi_{[0,t]}^{-1} \circ m_n^{-1}) \circ \pi_j^{-1} $ is given by
\begin{flalign}
 \overline \mu_{n,j}(A) = \hspace{6cm} \\
\notag
\int_{\R^{k+2}} \1_A(s_0,s_1,\dots,s_k,s_{k+1})
\sqrt{\frac{n}{2\pi (t-t_k)}}\,e^{-n \frac{(s_{k+1}- s_k)^2}{2(t-t_{k})}}
\prod_{i=1}^k \left[ \sqrt{\frac{n}{{2\pi (t_i - t_{i-1} )}}} \,
e^{-n \frac{(s_{i}- s_{i-1})^2}{2(t_{i}-t_{i-1})}} \right] \hspace{-5cm}  \\ \notag
  \qquad \times \frac{1}{Z_n} e^{-n \left[ V(s_0)
+ \frac{s_0^2}{2}\right]} \D s_{k+1} \D s_k \cdots \D s_1 \D s_0
&& (A\in \cB(\R^{k+2})).
\end{flalign}
Then $\rho_{n,t,j}\colon\, \R \times \cB(\R^{k+1}) \rightarrow [0,1]$ defined by
\begin{flalign}
&& & \rho_{n,t,j}(\alpha, A) = \\
&& &
\notag
\frac{
\int_{\R^{k+1}} \1_A(s_0,s_1,\dots,s_k)
e^{-n \frac{(\alpha- s_k)^2}{2(t-t_{k})}} \prod_{i=1}^k
\left[  e^{-n \frac{(s_{i}- s_{i-1})^2}{2(t_{i}-t_{i-1})}} \right]
e^{-n \left[ V(s_0) + \frac{s_0^2}{2}\right]} \D s_k \cdots \D s_1 \D s_0
}{
\int_{\R^{k+1}}
e^{-n \frac{(\alpha- s_k)^2}{2(t-t_{k})}} \prod_{i=1}^k
\left[  e^{-n \frac{(s_{i}- s_{i-1})^2}{2(t_{i}-t_{i-1})}} \right]
e^{-n \left[ V(s_0) + \frac{s_0^2}{2}\right]}  \D s_k \cdots \D s_1 \D s_0
} \hspace{-3cm} \\
\notag
&& &
&& (A\in \cB(\R^{k+1})),
\end{flalign}
is the weakly continuous proper regular conditional probability under $\overline \mu_{n,j}$ given the
coordinate at time $t$. 
By Kolmogorov's Theorem (e.g.\ Bogachev~\cite[Theorem 7.7.2]{Bo07}), there exists a measure
$\rho_{n,t}(\alpha,\cdot)$ on $C([0,t),\R)$ (see e.g.\ \cite[Theorem 7.7.4]{Bo07}, it is similar to the fact that the Brownian motion is a process on $C([0,t],\R)$, which is stated below \cite[Theorem 7.7.4]{Bo07}) for which $\rho_{n,t}(\alpha,\cdot)
\circ \pi_j^{-1} = \rho_{n,t,j}(\alpha,\cdot)$ for all $j\in \fJ$. Because $\alpha\mapsto \rho_{n,t,j}
(\alpha,\cdot)$ is strongly continuous for all $n\in\N$ and $j\in \fJ$ (see Appendix~\ref{appendix2}),
the map $\alpha \mapsto \rho_{n,t}(\alpha,\cdot)$ is (strongly and hence) weakly continuous, i.e.,
$\rho_{n,t}$ is the weakly continuous proper regular conditional probability of $\mu_n \circ \pi_{[0,t]}^{-1} \circ m_n^{-1}$ under
$\pi_t$. \\
$\bullet$ \emph{Proof of (b), large deviation principle}. \\
Let $j\in \fJ$ be given by $j= (t_0,t_1,\dots,t_k)$ and let $\alpha \in \R$. 
By den Hollander~\cite[Theorem III.17]{dH00}, the sequence $(\rho_{n,t,j}(\alpha,\cdot))_{n\in\N}$ satisfies
the large deviation principle with rate $n$ and rate function 
$I_j\colon\, \R^{k+1} \rightarrow [0,\infty]$
given by 
\begin{align}
\label{eqn:rate_function_finite_dimensionals}
(s_0,s_1,\dots,s_k) \mapsto
V(s_0) + \frac{s_0^2}{2} + \left[\sum_{i=1}^k
 \frac{(s_{i}- s_{i-1})^2}{2(t_{i}-t_{i-1})} \right] + \frac{(\alpha- s_k)^2}{2(t-t_{k})} - \fC_j,
\end{align}
where $\fC_j$ is such that \eqref{eqn:rate_function_finite_dimensionals} has infimum $0$, i.e., 
\begin{align}
\fC_j=\inf_{s_0,s_1,\dots,s_k\in \R} \left( V(s_0) + \frac{s_0^2}{2} + \left[\sum_{i=1}^k
 \frac{(s_{i}- s_{i-1})^2}{2(t_{i}-t_{i-1})} \right] + \frac{(\alpha- s_k)^2}{2(t-t_{k})} \right).
\end{align}
We will show that $I_j$ is a good rate function, i.e., $I_j$ has compact level sets. Let $c>0$. Let $K_0 = \{s_0 \in \R: V(s_0) + \frac{s_0^2}{2} \le c\}$, $K_i= \{s_i\in \R: \sup_{s_{i-1}\in K_{i-1}}  \frac{(s_{i}- s_{i-1})^2}{2(t_{i}-t_{i-1})}\le c\}$ for $i\in \{1,\dots,k\}$ and $K_{k}^*= \{s_k \in \R: \frac{(\alpha- s_k)^2}{2(t-t_{k})} \le c\}$. All these sets are compact and therefore also the set $\{(s_0,s_1,\dots,s_k)\in \R^{k+1}: s_i\in K_i \mbox{ for } i\in \{0,\dots,k-1\}, s_k \in K_k \cap K_k^*\}$. Since the level sets of $I_j$ are closed, we conclude by this that $I_j$ has compact level sets. \\
We will show that the constant $\fC_j$, does not depend on $j$, by showing 
\vspace{-2mm}
\begin{align}
\fC_j = \fC_{(0,t)} = \inf_{s_0 \in \R} V(s_0) + \frac{s_0^2}{2} 
 + \frac{(\alpha- s_0)^2}{2t} = C_{t,\alpha}.
\end{align}
First, note that $\frac{(a+b)^2}{c+d} \le \frac{a^2}{c} + \frac{b^2}{d}$ for $a,b,c,d\in \R$ with $c,d> 0$, since $(da-cb)^2\ge 0$. By this we conclude that $ [\sum_{i=1}^k
 \frac{(s_{i}- s_{i-1})^2}{2(t_{i}-t_{i-1})} ] + \frac{(\alpha- s_k)^2}{2(t-t_{k})} \ge \frac{(\alpha- s_0)^2}{2t}$ for all $s_0,s_1,\dots,s_k\in \R$ and thus that $\fC_j \ge \fC_{(0,t)}$. By letting $s_i = \psi(t_i)$ for $i\in \{1,\dots,k\}$, where $\psi(s) = s_0 + \frac{\alpha-s_0}{t} s$, we get $\left[\sum_{i=1}^k
 \frac{(s_{i}- s_{i-1})^2}{2(t_{i}-t_{i-1})} \right] + \frac{(\alpha- s_k)^2}{2(t-t_{k})} = \frac{(\alpha- s_0)^2}{2t}$. Hence we conclude $\fC_j= \fC_{(0,t)}$. \\
By the Dawson-G\"artner projective limit theorem
\cite[Theorem 4.6.1]{DeZe10} the sequence $(\rho_{n,t}(\alpha,\cdot))_{n\in\N}$
satisfies the large deviation principle on $\R^{[0,t)}$, equipped with the product topology (see the beginning of the proof \cite[Theorem 5.1.6]{DeZe10} why one can replace the projective limit by this product space) with rate $n$ and rate function $\R^{[0,t)}\rightarrow [0,\infty]$
given by $\phi \mapsto \sup_{j\in \fJ} I_j( \pi_j(\phi))$, i.e., 
\begin{align}
\label{eqn:sup_rate_function_paths}
\phi \mapsto \ &V(\phi(0)) + \frac{\phi(0)^2}{2} -C_{t,\alpha} + \\
\notag & \sup \Bigg \{ \left[\sum_{i=1}^k
 \frac{(\phi(t_{i})- \phi(t_{i-1}) )^2}{2(t_{i}-t_{i-1})} \right] + \frac{(\alpha- \phi(t_k))^2}{2(t-t_{k})} 
  :
 k\in \N, \ 0 <t_1< \cdots <t_k <t \Bigg \}.
\end{align}
Note that if $\phi \in \cA\cC([0,t),\R)$ and $\phi(s)$ does 
not converge to $\alpha$ as $s\uparrow t$, then $\sup_{j\in \fJ} I_j(\phi)=\infty$, 
since $\sup_{s \in (0,t)} \frac{(\alpha-\phi(s))^2}{2(t-s)} = \infty$. 
Furthermore, if $\phi \in \cA\cC([0,t),\R)$ and $\lim_{s\uparrow t} \phi(s) = \alpha$, then the function $\overline \phi : [0,t] \rightarrow \R$ given by $\overline \phi = \phi $ on $[0,t)$ and $\overline \phi(t) =\alpha$ is an element of $\cA \cC ([0,t],\R)$ and the supremum on the second line in \eqref{eqn:sup_rate_function_paths} is equal to 
\begin{align}
\sup \Bigg \{ \sum_{i=1}^{k+1}
 \frac{(\phi(t_{i})- \phi(t_{i-1}) )^2}{2(t_{i}-t_{i-1})}  
  :
 k\in \N, \ 0 <t_1< \cdots <t_k <t_{k+1}=t \Bigg \},
\end{align}
In \cite[Proof of Lemma 5.1.6]{DeZe10} (with $\Lambda^*(x)$ replaced by $x^2$) it is shown that this supremum is equal to $\frac12 \int_0^t \dot \phi^2(s) \D s$. Furthermore, in \cite[Proof of Lemma 5.1.6]{DeZe10} (last part) it is also shown that \eqref{eqn:sup_rate_function_paths} equals $\infty$ when $\phi\notin \cA\cC([0,t),\R)$. 
Hence $(\rho_{n,t}(\alpha,\cdot))_{n\in\N}$ satisfies the large deviation principle on $\R^{[0,t)}$ with rate $n$ and rate function $\R^{[0,t)} \rightarrow [0,\infty]$ given by \eqref{eqn:rate_function_paths}. 
This leaves us to prove that the large deviation principle also holds on $C([0,t),\R)$ equipped with the topology of uniform convergence. 
To prove this, by \cite[Theorem 4.1.5(b) and Theorem 4.2.6]{DeZe10}, it is sufficient to show that $(\rho_{n,t}(\alpha,\cdot))_{n\in\N}$ is exponentially tight in $C([0,t),\R)$ equipped with the topology of uniform convergence. 
The exponential tightness follows in turn by showing that the large deviation rate function in \eqref{eqn:rate_function_paths}, which we call $J$ here, has compact level sets in the uniform topology of $C([0,t),\R)$, i.e., for $b>0$ the set 
\begin{align}
K_b = \{ \phi\in \cA\cC([0,t),\R) \ : \ \lim_{s\uparrow t} \phi(s)=\alpha, 
\ J(\phi)  \le b\}
\end{align}
is compact. 
We prove this by using the Arzel\`a-Ascoli theorem (see e.g. \cite[Theorem C.8]{DeZe10}). Since $K_b$ is closed it is sufficient to show that $K_b$ is bounded and equicontinuous (actually the Arzel\`a-Ascoli theorem can not directly be used since $[0,t)$ is not compact, however proving that $\tilde K_b = \{\phi \in \cA\cC([0,t],\R) \ : \ \phi(t)=\alpha, \tilde J(\phi)\}$ is bounded and equicontinuous in $C([0,t],\R)$ suffices, where $\tilde J$ is the canonical extension of $J$ to $\R^{[0,t]}$. The proof is similar as showing that $K_b$ is bounded and equicontinuous). \\
\emph{Equicontinuity of $K_b$}. For $\phi \in K_b$ and $u,v\in [0,t)$ with $u<v$  (applying Jensen's inequality)
\begin{align}
\left(\frac{\phi(u)-\phi(v)}{u-v}\right)^2 \le \frac{1}{(u-v)} \int_u^v \dot \phi(s)^2 \D s \le  \frac{2b+C_{t,\alpha} }{(u-v)}.
\end{align}
and since $2m |x| \le x^2 +m^2 $ for all $m>0$ we have 
\begin{align}
|\phi(u)- \phi(v) | \le \frac{2b+C_{t,\alpha} }{2m} |u-v| + \frac{m}{2}.
\end{align}
This implies equicontinuity. \\
\emph{Boundedness of $K_b$}. 
Let $\psi_1,\dots,\psi_k$ be the global minimisers of the (lower semicontinuous) rate function $J$. By the proof of (c) we know that $\psi_i$ is the linear function that connects $(0,\psi_i(0))$ and $(t,\alpha)$, i.e., $\psi_i(s) = \psi_i(0) + \frac{\alpha - \psi_i(0)}{t} s$ for $s\in [0,t)$. 
Let $m>0$ be such that $|\psi_i(0)|\le m$ for all $i\in \{1,\dots,k\}$ and $|\alpha|\le m$ and such that $V(s_0) + \frac{s_0^2}{2}  + \frac{(\alpha- s_0)^2}{2t} \ge C_{t,\alpha} +b$  for $s_0 \in [-m,m]^c$. 
Suppose that $\phi\in \cA\cC([0,t),\R)$ with $\lim_{s\uparrow t} \phi(s) = \alpha$ and $\|\phi\|_\infty \ge m+ b+1$ (where $\|\cdot\|_\infty$ is the supremum norm). 
Let $u\in (0,t)$ be such that $|\phi(u)| \ge m+b+1$. 
Then the optimal path $\psi$ from $0$ to $t$ which agrees with $\phi$ in $0$, in $u$ and in $t$ (i.e., $\lim_{s\uparrow t} \psi(s) = \alpha$)  is the linear interpolation between the points $(0,\phi(0)), (u,\phi(u)), (t,\alpha)$ (see the proof of (c)), i.e.,
\begin{align}
\psi(s) = \begin{cases}
\phi(0)+ \frac{\phi(u)-\phi(0)}{u} s & \quad s\in [0,u], \\
\phi(u)+ \frac{\alpha - \phi(u)}{t-u}(s-u) & \quad s\in [u,t].
\end{cases}
\end{align}
So then we have 
\begin{align}
J(\phi) \ge J(\psi) \ge I_{(0,u,t)}(\psi) = V(\phi(0)) + \frac{\phi(0)^2}{2} - C_{t,\alpha} +  \frac{(\phi(u)-\phi(0))^2}{2u} + \frac{(\alpha-\phi(u))^2}{2(t-u)} \ge b,
\end{align}
since either $\phi(0) \in [-m,m]$ and thus $|\phi(0)- \phi(u)|^2 \ge (b+1)^2\ge b$ or $\phi(0) \in [-m,m]^c$ and thus $V(s_0) + \frac{s_0^2}{2}  + \frac{(\alpha- s_0)^2}{2t} \ge C_{t,\alpha} +b$. By this we conclude that the set $K_b$ bounded in $\|\cdot\|_\infty$-norm by $m+b+1$. 
\end{note}


\section{Proper weakly continuous regular conditional probabilities}
\label{appendix2}

\begin{definition}
\label{def:proper_weakly_cont_reg_cond_prob}
Let $\cX$ and $\cY$ be topological spaces with Borel sigma-algebras $\cB(\cX)$ and $\cB(\cY)$.
Equip $\cX \times \cY$ with the product topology. Then $\cB(\cX \times \cY) = \cB(\cX) \otimes \cB(\cY)$ 
(i.e., the smallest sigma-algebra containing all sets $A\times B$ with $A\in \cB(\cX)$ and $B\in \cB(\cY)$). 
Let $\mu$ be a probability measure on $\cB(\cX\times \cY)$ and let $\pi\colon\, \cX \times \cY \rightarrow 
\cY$ the canonical projection. Then $\gamma\colon\, \cY \times \cB(\cX) \rightarrow [0,1]$ is called a 
\emph{regular conditional probability} under $\mu$ of the first coordinate given the second, when 
$\gamma$ is a transition kernel and
\begin{flalign}
&&
\mu( A \times B) =\int \1_B(y) \gamma(y, A) \D \, [\mu\circ \pi^{-1}](y)
&& \big(A\in \cB(\cX), B\in \cB(\cY)\big).
\end{flalign}
$\gamma$ is called \emph{proper} when $\gamma(y,\cdot) =0$ for all $y\in \supp(\mu\circ \pi^{-1})^c$, where
\begin{align}
\supp(\nu) = \cY \setminus \bigcup\big\{ U\subset \cY\colon\, U \mbox{ is open and } \nu(U) =0\big\}
\end{align}
for measures $\nu$ on $\cB(\cY)$. $\gamma$ is called \emph{weakly continuous} when the map 
$\alpha \rightarrow \gamma(\alpha,\cdot)$ is weakly continuous.
\end{definition}

\begin{lemma}
With the notation as in Definition {\rm \ref{def:proper_weakly_cont_reg_cond_prob}}, if $\gamma_1,
\gamma_2\colon\, \cY \times \cB(\cX) \rightarrow [0,1]$ are two proper regular conditional probabilities under $\mu$ of the first coordinate given the second,
then  $\gamma_1(y,\cdot)= \gamma_2(y,\cdot)$ for $\mu\circ \pi^{-1}$-a.e.\ $y\in Y$. Consequently, if
there exists a weakly continuous proper regular conditional probability of $\mu$ under $\pi$, then it is
unique.
\end{lemma}

\begin{proof}
The first statement can be found in Bogachev \cite[Section 10.4]{Bo07}.
The second statement follows from the fact that if $\gamma_1$ and $\gamma_2$ are proper regular
conditional probabilities, then $\mu(B)=1$ for $B=\{y\in \supp(\mu)\colon\, \gamma_1(y,\cdot)
=\gamma_2(y,\cdot)\}$, and hence  $B$ is dense in $\supp(\mu)$. So if $\gamma_1$ and $\gamma_2$
are weakly continuous, then $B=\supp(\mu)$, i.e., $\gamma_1=\gamma_2$.
\end{proof}

We will use the following lemma to conclude that regular conditional probabilities with a continuous
bounded density are weakly continuous. This lemma is an easy consequence of Lebesgue's Dominated
Convergence Theorem.

\begin{lemma}
\label{lemma:mixing_with_function_gives_weakly_continuous_kernel}
Let $\cX$ and $\cY$ be topological spaces with Borel sigma-algebras  $\cB(\cX)$ and $\cB(\cY)$.
Let $\mu$ be a probability measure on $\cB(\cX)$. Let $f \in C_b(\cX\times \cY,\R)$. If $\gamma\colon\,
\cY \times \cB(\cX) \rightarrow [0,1]$ is given by
\begin{flalign}
&& \gamma(y,A) = \frac{\int \1_A(x) f(y,x) \D \mu(x)}{\int  f(y,x) \D \mu(x)}
&& \big(y\in \cY, A\in \cB(\cX)\big),
\end{flalign}
then $\gamma$ is weakly continuous (even strongly continuous, i.e., $y \mapsto \gamma(y,A)$
is continuous for all $A\in \cB(\cX)$).
\end{lemma}


\end{document}